\documentclass[preprint 12pt]{elsarticle}
\usepackage{bbm}

\usepackage{amssymb}
\usepackage{amsfonts}
\usepackage{mathrsfs}
\usepackage{amssymb}
\usepackage{color, amsmath,amssymb, amsfonts, amstext,amsthm, latexsym}

\usepackage{amssymb, epsfig, amssymb, latexsym}
\usepackage{amsmath}
\usepackage{graphicx}
\usepackage{longtable}

\allowdisplaybreaks

\numberwithin{equation}{section}

\newtheorem{theorem}{Theorem}[section]
\newtheorem{lemma}{Lemma}[section]
\newtheorem{remark}{Remark}[section]

\newtheorem{proposition}{Proposition}[section]

\journal{arXiv}

\begin{document}

\begin{frontmatter}



\title{Weak order in  averaging principle for two-time-scale stochastic partial differential equations }


\author{Hongbo Fu}
\ead{hbfuhust@gmail.com}
\author{Li Wan}
\ead{wanlinju@aliyun.com}
\address{College of Mathematics and Computer Science, Wuhan Textile University, Wuhan, 430073, PR China}

\author{Jicheng Liu\corref{cor1}}   
\ead{jcliu@hust.edu.cn}
\author{Xianming Liu}
\ead{xmliu@hust.edu.cn}
\address{School of Mathematics and Statistics, Huazhong University of Science and
Technology, Wuhan, 430074, PR China} 

\cortext[cor1]{Corresponding author at: School of Mathematics and
  Statistics, Huazhong University of Science and Technology, Wuhan,
      430074, China.  }
\begin{abstract}

This work is devoted to averaging principle of a two-time-scale
stochastic partial differential equation  on a bounded interval $[0,
l]$, where both the fast and slow components are directly perturbed
by additive noises. Under some regular conditions on drift
coefficients, it is proved that the rate of weak convergence for the
slow variable to the averaged dynamics is of order $1-\varepsilon$
for arbitrarily small $\varepsilon>0$. The proof is based on  an
asymptotic expansion of solutions to Kolmogorov equations associated
with the multiple-time-scale system.

\end{abstract}


\end{frontmatter}



\section{Introduction}
In a previous paper \cite{Brehier}, Br\'{e}hier exhibited the strong
and weak order of an averaging principle for  the following  class
of slow-fast stochastic reaction-diffusion equation  on a bounded
 interval $D=[0, l]$ of $\mathbb{R}$:
\begin{eqnarray}
\begin{cases}
 \frac{\partial }{\partial
t}x^\epsilon_t(\xi)= \Delta x^\epsilon_t(\xi)
+F(x^\epsilon_t(\xi),y^\epsilon_t(\xi)), \;\xi\in D,\;t>0,\\
 \frac{\partial }{\partial
t}y^\epsilon_t(\xi)=\frac{1}{\epsilon} \Delta y^\epsilon_t(\xi)
+\frac{1}{\epsilon}G(x^\epsilon_t(\xi),y^\epsilon_t(\xi))+\frac{1}{\sqrt{\epsilon}}
\dot{W}_t(\xi),\;\xi\in D,\;t>0,\\
x_0^\epsilon(\xi)=x(\xi),\;y_0^\epsilon(\xi)=y(\xi), \;\xi\in D, \\
x^\epsilon_t(0)=x^\epsilon_t(l)=0,\; t\geq0,\\
y^\epsilon_t(0)=y^\epsilon_t(l)=0,\; t\geq0,
\end{cases}\label{eqation-orignal-0}
\end{eqnarray}
where  the leading linear operator
$\Delta=\frac{\partial^2}{\partial\xi^2}$ is the Laplacian operator.
The ratio of time-scale separation is described by the positive and
small parameter $\epsilon$. With this time scale, the process
$x_t^\epsilon(\xi)$ is always called as the slow component and
$y^\epsilon_t(\xi)$ as the fast component.  The drift coefficients
$F$ and $G$ are suitable mappings from $ L^2(D)$ to itself. The
stochastic perturbation $W_t$ is an $L^2(D)$-valued Wiener process
with respect to a filtered probability space $(\Omega, \mathscr{F},
\mathscr{F}_t, \mathbb{P})$.

In many applications, it is of interest to describe dynamics of the
slow variable. Since exact solution  is hard to be known,   a simple
equation without fast component, which can capture the essential
dynamics of slow variable, is highly desirable. The fundamental
method to approximate slow solution $x^\epsilon_t(\xi)$ to equation
(1.1) is the so-called averaging procedure. Under some conditions,
it has been proven that the slow solution $x^\epsilon_t(\xi)$ to
problem \eqref{eqation-orignal-0} converges (as $\epsilon$ tends to
$0$), in a suitable sense, to   solution   of the  so-called
averaged equation, obtained by eliminating fast component via taking
the average of the coefficient $F$ over the slow equation.

 Once  one has proved the validity of averaging principle, a critical question
arises as to how do we   determine the rate of convergence for this
procedure.  In Br\'{e}hier \cite {Brehier}, it has been proved that
the strong convergence (approximation in pathwise) order is
$\frac{1}{2}-\varepsilon$ while the weak convergence (approximation
in law) order is $1-\varepsilon$, for arbitrarily small
$\varepsilon>0$, on condition that the slow motion equation is a
deterministic parabolic equation. Due to the arbitrariness of
$\varepsilon$ we may say that  strong (resp. weak) convergence order
is $\frac{1}{2}^-$ (resp. $1^-$). If an additive noise is included
in the slow motion equation, the strong convergence order will
decrease to $\frac{1}{5}$. In this case, unfortunately, the methods
used to prove the weak order will be invalid. The main difficulty is
due to the fact that tactics depend on the time derivative of
solution to the averaged equation, which does not exist in any
general way with the case where the slow motion equation is
perturbed with a noise. In a more recent work following the
procedure inspired by Br\'{e}hier \cite {Brehier}, Dong et al.
\cite{Dong}  establish weak order $1^-$ in stochastic averaging for
one dimensional Burgers equation only in the particular case of an
additive noise on the fast component.

Unlike in   the above-mentioned papers, where  the noise acts only
in the fast motion, in the present paper we study a class of
stochastic partial differential equations on the bounded interval
$D=[0, l]$ of $\mathbb{R}$, involving two separated time scales,
which can be written as:
\begin{eqnarray}
\begin{cases}
 \frac{\partial }{\partial
t}x^\epsilon_t(\xi)= \Delta x^\epsilon_t(\xi)
+F(x^\epsilon_t(\xi),y^\epsilon_t(\xi))+\sigma_1
\dot{W}^1_t(\xi), \;\xi\in D,\;t>0,\\
 \frac{\partial }{\partial
t}y^\epsilon_t(\xi)=\frac{1}{\epsilon} \Delta y^\epsilon_t(\xi)
+\frac{1}{\epsilon}G(x^\epsilon_t(\xi),y^\epsilon_t(\xi))+\frac{\sigma_2}{\sqrt{\epsilon}}
\dot{W}^2_t(\xi),\;\xi\in D,\;t>0,\\
x_0^\epsilon(\xi)=x(\xi),\;y_0^\epsilon(\xi)=y(\xi), \;\xi\in D, \\
x^\epsilon_t(0)=x^\epsilon_t(l)=0,\; t\geq0,\\
y^\epsilon_t(0)=y^\epsilon_t(l)=0,\; t\geq0.
\end{cases}\label{eqation-orignal}
\end{eqnarray}
 Assumptions on
regularity of the drift coefficients $F$ and $G$ will be given
below. The noises $W^1_t(\xi)$ and $W^2_t(\xi)$ are independent
Wiener processes which will be detailed in next section. The
coefficients of noise strength $\sigma_1$ and $\sigma_2$ are
positive constants. The coupled stochastic partial differential
equation in form of \eqref{eqation-orignal} arises from many
physical systems when random spatio-temporal forcing is taken into
account, such as diffusive phenomena in media, epidemic propagation
and transport process of chemical species.


So far, the  explicit order for weak convergence in averaging has
not be extended to the  general situation when  both the fast and
slow components are directly perturbed with some noises. In the
current article, we will show that   the {weak order} $1^-$ can be
achieved even when there is a noise in the slow motion equation.
More precisely, we prove that for any $T>0$ and  a class of test
functions $\phi: L^2(D)\rightarrow \mathbb{R}$, with continuous and
bounded derivatives up to the third order,
\begin{eqnarray}
|\mathbb{E}\phi( {x}^\epsilon_T)-\mathbb{E}\phi(\bar{X}_T)|\leq
C\epsilon^{1-r}\label{error}
\end{eqnarray}
for any $r\in(0,1)$, where $C$ is a constant independent of
$\epsilon$ (see Theorem \ref{theorem}). In the estimate above, the
averaged motion $\bar{X}_t$ solves the equation
\begin{eqnarray*}
 \begin{cases}
 \frac{\partial }{\partial
t }\bar{X}_t(\xi)=\Delta \bar{X}_t(\xi)+ \bar{F}(\bar{X}_t(\xi))+\sigma_1\dot{W}_t^{1}(\xi),\;\;\xi\in D,\;t>0,\\
 \bar{X}_0(\xi)=x(\xi), \;\xi\in D,  \\
 \bar{X}_t(0)=\bar{X}_t(l)=0, \; t\geq 0,
\end{cases}
\end{eqnarray*}
with an averaged drift  $\bar{F}(x):=\int_{L^2(D)}F(x,y)\mu^x(dy)$,
where $\mu^x$ is  the unique mixing invariant measure for fast
variable  with frozen slow component (see equation \eqref{frozen}).

In order to prove \eqref{error}, we adopt an asymptotic expansion
scheme  as in \cite{Brehier} to decompose
$\mathbb{E}\phi({x}^\epsilon_t)$ with respect to the scale parameter
$\epsilon$ in form
\begin{equation*}
\mathbb{E}\phi({x}^\epsilon_t)=u_0+\epsilon u_1+r^\epsilon,
\end{equation*}
 where the functions $u_0$, $u_1$ and $r^\epsilon$ are determined
 recursively and
 obey  certain  linear evolutionary equations.  First of all, we identify  leading term  $u_0$
as $\mathbb{E}\phi(\bar{X}_t)$ by a uniqueness argument. To this
purpose, we introduce the Kolmogorov operators with parameter  to
construct an evolutionary equation that describes both $u_0$ and
$\mathbb{E}\phi(\bar{X}_t)$. Moreover, this also allows us to
characterize the expansion coefficient $u_1$ by a Poisson equation
associated with the generator of fast process so that we obtain an
explicit expression of $u_1$. As a result, some a priori estimates
guarantee the boundedness of function $u_1$.

The next key step consists in estimate for the remainder term
described by a linear equation depending  on $\mathcal{L}_2u_1$  and
$\frac{\partial u_1}{\partial t}$, where $\mathcal{L}_2 $ is the
Kolmogorov operator for  slow motion equation with frozen fast
component. Due to the presence of unbounded operator
$\mathcal{L}_2$, we have to reduce the problem to its Galerkin
finite dimensional version. Since the noise is included in the slow
motion equation,  the It\^{o}  formula is employed to derive an
explicit expression for $\frac{\partial u_1}{\partial t}$, which is
related to the {third   derivative of} $\phi$.  This is the reason
why we have to require the test function to be  $3$-times
differentiable. After bounding the terms $\mathcal{L}_2u_1$ and
$\frac{\partial u_1}{\partial t}$, the remainder $r^\epsilon$ in the
expansions can be estimated by standard evolution equation method
and the weak error  with an explicit order  is achieved, where
It\^{o}'s formula is used again to overcome the non-integrability of
$r^\epsilon$ at zero point. We would like to stress that, due to the
regular conditions imposed on noise in slow component (see
\eqref{Trace} and \eqref{Tr-AQ}), the solution process to slow
equation enjoys values in the domain of dominating linear operator.
This allows estimates using techniques similar to those in
Br\'{e}hier \cite{Brehier}.

Up to now, to our knowledge, this is the first to obtain the weak
convergence order   for averaging of stochastic partial differential
equations in the case of a noise acting directly on the slow motion
equation. It is certainly believable that our method can be applied
to stochastic Burgers equation with regular noise such that weak
order $1^-$ in averaging is obtained. This will extend work of Dong
et al. \cite{Dong},  as we do not require the slow motion equation
is deterministic.

Averaging method plays a prominent role in the study of qualitative
behavior of dynamical systems with two time scales and   has a long
and rich history.  Their rigorous mathematical justification was due
to Bogoliubov \cite{Bogoliubov} for the deterministic dynamical
system. Further developments   to ordinary differential equations of
the averaging theory can be found in Volosov \cite{Volosov}, Besjes
\cite{Besjes} and Gikhman \cite{Gikhman}. The averaging  result  for
stochastic differential equations of It\^{o} type was firstly
introduced in Khasminskii \cite{Khas},  being an extension of the
theory to stochastic case. Since then, much progress has been made
for multiple-time-scale stochastic dynamical systems in finite
dimensions, see for instance \cite {Freidlin-Wentzell1,
Freidlin-Wentzell2, Givon1, Khas2, Khas3, Kifer1, Kifer2, Kifer3,
LiuDi, Vere1, Vere2, Wainrib, Weinan}
 and the references therein.
 In particular, averaging for finite dimensional stochastic systems with non-Gaussian
 noise may be found in \cite{Xu, Xu2, Xu3, Xu4}.
 In a series of recent papers, Cerrai and Freidlin \cite{Cerrai1}  and Cerrai \cite{Cerrai2}
  studied an infinite
 dimensional version of averaging principle for   partial differential equations
of reaction-diffusion type with additive and multiplicative Wiener
noise, respectively, where global Lipschitz conditions were imposed.
In contrast to Lipschitz setting, averaging principle for parabolic
equations with polynomial growth coefficients was explored in Cerrai
\cite{Cerrai-Siam}. For the extensions to stochastic parabolic
equations with non-Gaussian  stable noise, we are referred to Bao et
al. \cite{Bao}. For related works on averaging for infinite
dimensional stochastic dynamical systems we refer the reader to
\cite{wangwei, Fu-Liu, Fu-Liu-2, Pei, Thompson, Hogele}.


The rest of the paper is arranged as follows.  Section 2 is devoted
to the general notation and framework. The ergodicity of fast
process and the main result are introduced in Section 3. Then  some
a priori estimates is presented in Section 4. In Section 5, we
present an asymptotic expansion scheme. In the final section, we
state and prove technical lemmas applied in the preceding section.

Throughout the paper, the letter $C$ below with or without
subscripts will denote positive constants whose value may change in
different occasions. We will write the dependence of constant on
parameters explicitly if it is essential.

\section{Notations and preliminary results}\label{notations}
To rewrite the system \eqref{eqation-orignal} as an abstract
evolutionary equation, we present some notations and recall some
well-known facts for later use.

For a fixed domain $D=[0, l]$, let $H$ be the real, separable
Hilbert space $L^2(D)$, endowed with the usual scalar product
$\Big(\cdot, \cdot\Big)_H$. The corresponding norm is denoted by
$\|\cdot\|$. Denote by $\mathcal{L}(H)$ the Banach space  of linear
and bounded operators from $H$ to itself, equipped with usual
operator norm.

Let $\{ e_k(\xi)\}_{k\geq 1}$ denote the complete orthonormal system
of eigenfunctions in $H$ such that, for $k = 1,2,\ldots$,
\begin{equation*}\label{eigenfunction} -\Delta
e_k(\xi)=\alpha_ke_k(\xi),\;\;e_k(0)=e_k(l)=0 ,
\end{equation*}
with $0<\alpha_1\leq\alpha_2\leq\cdots\alpha_k\leq\cdots$.  We would
like to recall the fact that
$e_k(\xi)=\sqrt{\frac{2}{l}}\sin\frac{k\pi\xi}{l}$ and
$\alpha_k=-\frac{k^2\pi^2}{l^2}$ for $k = 1,2,\cdots$.

Let $A$ be the Laplacian operator $\Delta$ satisfying zero Dirichlet
boundary condition, with domain $\mathscr{D}(A)=H^{ 1}_0(D)\cap
H^2(D)$, which generates a strongly continuous semigroup
$\{S_t\}_{t\geq 0}$ on $H$, defined by, for any $h\in H$,
\begin{eqnarray*}
S_th=\sum\limits_{k\in \mathbb{N}} e^{-\alpha_kt}e_k\Big(e_k,
h\Big)_H.
\end{eqnarray*}
Here, for the sake of brevity, we omit to write the dependence of
the spatial variable $\xi$. It is straightforward to check that
$\{S_t\}_{t\geq0}$ is a contractive semigroup  on $H$. For
$\gamma\in [0,1]$ we defined the operator $(-A)^\gamma$ by
\begin{eqnarray*}
(-A)^\gamma x=\sum\limits_{k\in \mathbb{N}}\alpha_k^\gamma x_ke_k\in
H
\end{eqnarray*}
 with domain
 \begin{eqnarray*}
 \mathscr{D}((-A)^\gamma)=\left\{x=\sum\limits_{k\in\mathbb{N}}x_ke_k\in H; \;
 \|x\|^2_{(-A)^\gamma}:=\sum\limits_{k\in\mathbb{N}}\alpha_k^{2\gamma}x_k^2<\infty
 \right\}.
 \end{eqnarray*}
Using the spectral decomposition of $A$, the semigroup
$\{S_t\}_{t\geq 0}$ enjoys the following smooth  property.
\begin{proposition}\label{proposition}
For any $\gamma\in [0, 1]$ there exists a constant $C_\gamma>0$ such
that
\begin{eqnarray}
&&\!\!\!\!\!\!\|S_tx\|_{(-A)^\gamma}\leq C_\gamma
t^{-\gamma}e^{-\frac{\alpha_1}{2}t}\|x\|, \;t>0, x\in H,\label{propo-1}\\
&&\!\!\!\!\!\!\|S_tx-S_\tau x\|\leq C_\gamma
\frac{|t-\tau|^\gamma}{\tau^\gamma}e^{-\frac{\alpha_1}{2}\tau}\|x\|,\;t,
\tau>0,x \in H,\label{propo-2}\\
&&\!\!\!\!\!\!\|S_tx-S_\tau x\|\leq C_\gamma |t-\tau|^\gamma
e^{-\frac{\alpha_1}{2}\tau}\|x\|_{(-A)^\gamma},\;t,\tau>0, x\in
\mathscr{D}((-A)^\gamma).\label{propo-3}
\end{eqnarray}
\end{proposition}
For the  perturbation noises we suppose the following setting. For
$i=1,2$, let $W_t^i$  be the Wiener process  on a  stochastic base
$(\Omega, \mathscr{F}, \mathscr{F}_t, \mathbb{P})$  with a bounded
covariance operator $Q_i: H\rightarrow H$ defined by
$Q_ie_k=\lambda_{i,k}e_k$, where $\{\lambda_{i,k}\}_{k\in
\mathbb{N}}$ are nonnegative and $\{e_k\}_{k\in \mathbb{N}}$ is the
complete orthonormal basis in $H$. Formally, for $i=1,2,$ Wiener
processes $W^i_t$ can be written  as the infinite sums (cf. Da Prato
and Zabczyk \cite{Daprato})
\begin{eqnarray*}
W_t^i=\sum\limits_{k\in\mathbb{N}}\sqrt{\lambda_{i,k}}B^{(i)}_{t,k}e_k,
\end{eqnarray*}
where $\{B^{(i)}_{t,k}\}_{k\in \mathbb{N}}$ are mutually independent
real-valued Brownian motions on stochastic  base $(\Omega,
\mathscr{F}, \mathscr{F}_t, \mathbb{P})$.  For the sake of
simplicity we prefer to assume that both $Q_1$ and $Q_2$ have finite
trace, that is  there exists a positive constant $C$ such that
\begin{eqnarray}
Tr(Q_i)=\sum\limits_{k\in \mathbb{N}} {\lambda_{i,k}}\leq C,\;i=1,2.
\label{Trace}
\end{eqnarray}
Moreover, we also assume
\begin{eqnarray}
Tr\big((-A)Q_1\big)=\sum\limits_{k\in\mathbb{N}} \lambda_{1,k}
\alpha_k\leq C.\label{Tr-AQ}
\end{eqnarray}
Concerning the drift coefficients $F$ and $G$ we shall
impose the following conditions.\\

(H.1) For each fixed $u\in H$, the mapping $F(u,\cdot):H\rightarrow
H$ is of a class $\mathcal {C}^3$, with bounded derivatives,
uniformly with respect $u\in H$. Also suppose that for any $v\in H$,
the mapping $F(\cdot,v): H\rightarrow H$ is of class $\mathcal
{C}^3$, with   bounded derivatives, uniformly for $v\in H$.\\

 (H.2) For each fixed $u\in H$, the mapping $G(u,\cdot):H\rightarrow
H$ is of a class $\mathcal{C}^2$, with bounded derivatives,
uniformly with respect $u\in H$. Also suppose that for any $v\in H$,
the mapping $G(\cdot,v): H\rightarrow H$ is of class
$\mathcal{C}^2$, with bounded derivatives, uniformly with respect
$v\in H$. Moreover, we assume that
\begin{eqnarray*}
\sup\limits_{u\in H}\|G_v'(u,v)\|_{\mathcal{L}(H)}:=L_g<\alpha_1,
\end{eqnarray*}
where $G_v'$ denotes the derivative with respect to $v$ and
$\|\cdot\|_{\mathcal{L}(H)}$ denotes the operator norm on
$\mathcal{L}(H)$.
\begin{remark}
Under (H.1) and (H.1),  it is not difficult to verify that there
exist positive constants $K_F$ and $K_G$ such that
\begin{eqnarray}
\|F(u_1,v_1)-F(u_2,v_2)\|\leq
K_F(\|u_1-u_2\|+\|v_1-v_2\|),\;u_1,u_2, v_1,v_2\in H,
\label{F-condi-1}
\end{eqnarray}
and
\begin{eqnarray}
\|G(u_1,v_1)-G(u_2,v_2)\|\leq
K_G(\|u_1-u_2\|+\|v_1-v_2\|),\;u_1,u_2, v_1,v_2\in H,\label{g-condi}
\end{eqnarray}
which means $F, G: H\times H\rightarrow H$ are Lipschitz continuous.
\end{remark}

Once introduced the main notations, system \eqref{eqation-orignal}
can be written as
\begin{eqnarray}
\begin{cases}
dX_t^\epsilon=AX^\epsilon_tdt+F(X^\epsilon_t,Y^\epsilon_t)dt+\sigma_1dW_t^1,\;X_0^\epsilon=x, \\
dY_t^\epsilon=\frac{1}{\epsilon}AY^\epsilon_tdt+
\frac{1}{\epsilon}G(X^\epsilon_t,Y^\epsilon_t)dt+\frac{\sigma_2}{\sqrt{\epsilon}}dW_t^2,\;Y_0^\epsilon=y.
\end{cases}\label{abstr-equation}
\end{eqnarray}
By virtue of conditions \eqref{F-condi-1} and \eqref{g-condi}, it is
easy to check that system \eqref{abstr-equation} admits a unique
mild solution, which, in order to emphasize the dependence on the
initial data, is denoted by $(X_t^\epsilon(x,y),Y_t^\epsilon(x,y))$.
This means that for any $t>0$, it holds $\mathbb{P}-a.s.$ that
\begin{eqnarray}
\begin{cases}
X_t^\epsilon(x,y)=S_tx+\int_0^tS_{t-s}F(X^\epsilon_s(x,y),Y^\epsilon_s(x,y))ds
+\sigma_1\int_0^tS_{t-s}dW_s^1,\\
Y_t^\epsilon(x,y)=S_{\frac{t}{\epsilon}}y+\frac{1}{\epsilon}\int_0^tS_\frac{t-s}{\epsilon}G(X^\epsilon_s(x,y),Y^\epsilon_s(x,y))ds
+\frac{\sigma_2}{\sqrt{\epsilon}}\int_0^tS_\frac{t-s}{\epsilon}dW_s^2.\label{equation-mild}
\end{cases}
\end{eqnarray}

Moreover, by using standard arguments, we have the following lemma.
\begin{lemma}\label{moment-bound}
Under (H.1) and (H.2), for any $T>0$ and $x, y\in H $, there exists
a positive constant $C_T$ such that for any $x, y\in H$ and
$\epsilon\in (0,1]$,
\begin{eqnarray*}
&&\sup\limits_{t\in[0, T]}\mathbb{E}\|X^\epsilon_t(x,y)\|^2\leq
C_T(1+\|x\|^2+\|y\|^2),\\
&&\sup\limits_{t\in[0, T]}\mathbb{E}\|Y^\epsilon_t(x,y)\|^2\leq
C_T(1+\|x\|^2+\|y\|^2).
\end{eqnarray*}
\end{lemma}

To study  weak convergence, we need to introduce some notations in
connection with the test function. If $\mathcal {X}$ is a Hilbert
space equipped with inner product $(\cdot,\cdot)_\mathcal {X}$, we
denote by $\mathcal{C}^k(\mathcal {X},\mathbb{R})$ the space of  all
$k-$times continuously Fr\'{e}chet differentiable functions
$\phi:\mathcal {X}\rightarrow \mathbb{R}, x\mapsto \phi(x)$. By
$\mathcal{C}_b^k(\mathcal {X},\mathbb{R})$ we denote the subspace of
functions from $\mathcal{C}^k(\mathcal {X},\mathbb{R})$  which are
bounded together with their derivatives. For $\phi\in
\mathcal{C}^m(\mathcal {X},\mathbb{R})$, we use the notation
$D^m_{\underbrace{xx\cdots x}_{m -{times}}}\phi(x)$ for its $m$-th
derivative   at the point $x$.

Thanks to Riesz representation isomorphism, we may get the identity
for $x,h\in \mathcal {X}$:
$$D_x\phi(x)\cdot h=(D_x\phi(x), h)_\mathcal {X}.$$
For $\phi\in \mathcal{C}^2(\mathcal {X}, \mathbb{R})$, we will
identify the second derivative $D^2_{xx}\phi(x)$ with a bilinear
operator from $\mathcal {X}\times \mathcal {X}$ to $\mathbb{R}$ such
that
$$D^2_{xx}\phi(x)\cdot (h,k)=(D^2_{xx}\phi(x)h,k)_\mathcal {X},\;\; x,h,k\in \mathcal {X}.$$
On some occasions, we also use the notations $\phi',\phi''$ and
$\phi'''$ instead of $D_{x}\phi$, $D_{xx}^2\phi$ and
$D_{xxx}^3\phi$, respectively.

\section{Ergodicity of $Y_t^x$ and averaging dynamics}
For fixed $x\in H$ consider the problem associate to fast motion
with frozen slow component
\begin{eqnarray}
\begin{cases}
dY_t^x=AY_t^xdt+G(x, Y_t^x)dt+\sigma_2dW^2_t,\\
Y_0^x=y.\label{frozen}
\end{cases}
\end{eqnarray}
Notice that the drift $G: H\times H\rightarrow H $ is Lipshcitz
continuous. By arguing as before, for any fixed slow component $x\in
H$ and any initial data $y\in H$, problem \eqref{frozen} has a
unique mild solution denoted by $Y_t^{x}(y)$. Now, we consider the
transition semigroup $P_t^x$ associated with  process $Y_t^x(y)$, by
setting for any $\psi \in \mathcal {B}_b(H)$ the space of bounded
functions on $H$,
\begin{equation*}
P_t^x\psi(y)=\mathbb{E}\psi(Y_t^x(y)).
\end{equation*}
By adopting a similar approach used in \cite{Fu-Liu}, we can  show
that
\begin{equation}\label{Fast-motion-energy-bound}
\mathbb{E}\|Y^x_t(y)\|^2\leq
C\left(e^{-(\alpha_1-L_g)t}\|y\|^2+\|x\|+1\right),\; t>0
\end{equation}
for some constant $C>0$. This implies the existence of an invariant
measure   $\mu^x $ for the Markov semigroup $P^x_t$ associated with
system \eqref{frozen} on $H$ such that
$$ \int_HP^x_t \psi d\mu =\int_H\psi d\mu^x , \quad
t\geq 0
$$
for any  $\psi \in \mathcal {B}_b(H)$ (for a proof, see, e.g.,
\cite{Cerrai2}, Section 2.1). We recall that in  \cite{Cerrai-Siam}
it is proved the invariant measure has finite $2-$moments:
\begin{equation}\label{mu-Momenent-bound}
\int_H\|y\|^2\mu^x(dy)\leq C(1+\|x\|^2).
\end{equation}
Let $Y_t^x(y')$ be the solution of problem \eqref{frozen} with
initial value $Y_0=y'$, it is not difficult to show that for any
$t\geq0$,
\begin{eqnarray}\label{initial-diff}
\mathbb{E}\|Y_t^x(y)-Y_t^x( y')\|^2\leq C\|y-y'\|^2e^{-\beta t}
\end{eqnarray}
with $\beta=(\alpha_1-L_g)>0,$ which implies that $\mu^x$ is the
unique invariant measure for $P^x_t$. This allows us to define an
$H$-valued mapping $\bar{F}$ by averaging the coefficient $F$ with
respect to the invariant measure $\mu^x$, that is,
\begin{equation*}
\bar{F}(x):=\int_HF(x,y)\mu^x(dy), x\in H, \label{aver-F}
\end{equation*}
and then, by using the  condition \eqref{F-condi-1}, it is immediate
to check that
\begin{eqnarray}
\|\bar{F}(x_1)-\bar{F}(x_2)\|\leq K_F\|x_1-x_2\|, \; x_1, x_2\in H.
\label{barF-lip}
\end{eqnarray}
According to the invariant property of $\mu^x$,
\eqref{mu-Momenent-bound} and  \eqref{F-condi-1}, we have
\begin{eqnarray}
\nonumber \|\mathbb{E}F(x, Y_t^{x}( y))-\bar{F}(x) \|^2 &=&
\|\int_{H}\mathbb{E}\big(F(x, Y_t^{x}
(y)-F(x, Y_t^{x}(z)\big)\mu^x(dz) \|^2 \\
\nonumber&\leq&\int_{H}\mathbb{E}\left\|Y_t^{x}
(y)-Y_t^{x}(z)\right\|^2 \mu^x(dz)\\
\nonumber&\leq&e^{-\beta t}\int_{H}\|y-z\|^2 \mu^x(dz)\\
&\leq&Ce^{-\beta t}\big(1+\|x\|^2 +\|y\|^2
\big),\label{Averaging-Expectation}
\end{eqnarray}
which means that
\begin{eqnarray}
\bar{F}(x)=\lim\limits_{t\rightarrow +\infty}\mathbb{E}F(x,
Y^x_t(y)),\;x\in H. \label{bar-F-lim}
\end{eqnarray}
Using this limit   and Assumptions (H.1), it is possible to show
that
\begin{eqnarray}\label{bar-F-derivative}
\|\bar{F}'(x)\cdot h\|\leq C\|h\|,\; x, h\in H.
\end{eqnarray}
Now we  introduce the  effective dynamics:
\begin{eqnarray*}
 \begin{cases}
 \frac{\partial }{\partial
t }\bar{X}_t(\xi)=\Delta \bar{X}_t(\xi)+ \bar{F}(\bar{X}_t(\xi))+\sigma_1\dot{W}_t^{1}(\xi),\;\xi \in D, t> 0,\\
 \bar{X}_t(0)=\bar{X}_t(l)=0, \;t\geq 0, \\
 \bar{X}_0(\xi)=x(\xi), \;\xi\in D.
\end{cases}
\end{eqnarray*}
By using the notations introduced in Section \ref{notations} it can
be written as an abstract evolutionary equation in form
\begin{eqnarray}
 \label{Averaging-equation}
 \begin{cases}
 d\bar{X}_t=A\bar{X}_tdt+\bar{F}(\bar{X}_t)dt+\sigma_1dW^1_t,\;t>0,\\
 \bar{X}_0=x.
\end{cases}
\end{eqnarray}
 For any
initial datum $x\in H$, the equation \eqref{Averaging-equation}
admits a unique mild solution, which means that there exists a
unique adapt process $\bar{X}_t(x)$ such that
\begin{eqnarray*}
\bar{X}_t(x)=S_tx+\int_0^tS_{t-s}\bar{F}(\bar{X}_s(x))ds+\sigma_1\int_0^tS_{t-s}dW^1_s,\;\mathbb{P}-a.s.,\;
t\geq0.
\end{eqnarray*}
Moreover, for any $T>0$ it can be easily proved  that
\begin{eqnarray}
\mathbb{E}\|\bar{X}_t(x)\|^2\leq C_T(1+\|x\|^2),\; t\in [0,
T].\label{bar-x-moment}
\end{eqnarray}
Thanks to averaging principle (see Cerrai \cite{Cerrai2} for
details), we have that the limit
\begin{eqnarray}
\lim\limits_{\epsilon\rightarrow 0+}\sup\limits_{0\leq t\leq
T}\mathbb{E}\|\bar{X}_t(x)-{X}^{\epsilon
}_t(x,y)\|^2=0\label{aver-infinite}
\end{eqnarray}
for any fixed $T>0.$ This means that the slow process
$X^\epsilon_t(x,y)$ enjoys strong convergence  to the averaging
process $\bar{X}_t(x)$. Moreover, the strong order in averaging is
$\frac{1}{5}^-$ (Br\'{e}hier  \cite{Brehier}). The weak convergence
using test functions is obvious.  Our aim is to establish rigorously
weak error bounds for the limit of slow process  with respect to
scale parameter $\epsilon$. The main result of this paper is the
following, whose proof is postponed in the end of Section
\ref{asym}.
\begin{theorem}\label{theorem}
Assume that $x\in \mathscr{D}((-A)^\theta)$ for some $\theta\in (0,
1]$ and $y \in H.$ Then, under (H.1)
 and (H.2), for any $r\in (0, 1)$, $T>0$ and $\phi\in C_b^3(H, \mathbb{R})$, there exists a
 constant $C_{\theta,r,T,\phi,x,y}$ such that
 \begin{eqnarray*}
 \left|\mathbb{E}\phi(X^\epsilon_T(x,y))-\mathbb{E}\phi(\bar{X}_T(x))\right|\leq
 C_{\theta,r,T,\phi,x,y}\epsilon^{1-r}.
 \end{eqnarray*}
\end{theorem}

\section{Some a priori estimates}
Before proving the main results, we need to state some technical
lemmas used in   subsequent section.
\begin{lemma}\label{Xt-Xs}
Let the conditions  (H.1)  and  (H.2)  be satisfied and fix $x,y\in
H$ and $T>0$. Then for any $r\in (0,\frac{1}{2})$ there exists a
constant  $C_{r,T}>0$ such that for any   $0< s\leq t\leq T$, we
have
\begin{eqnarray*}
&&(\mathbb{E}\|X_t^\epsilon(x,y)-X^\epsilon_s(x,y)\|^2)^{\frac{1}{2}}\\
&&\leq
C_{r,T}\frac{|t-s|^{1-r}} {s^{1-r}}\|x\|\\
&&+ C_{r,T}(|t-s|^{\frac{1}{2}}+|t-s|^{1-r}+|t-s|^r)(1+\|x\|+\|y\|).
\end{eqnarray*}
\begin{proof}
We can write
\begin{eqnarray}
X_t^\epsilon(x,y)-X_s^\epsilon(x,y)&=&(S_t -S_s)x+\int_s^tS_{t-\tau}F(X_\tau^\epsilon(x,y),Y^\epsilon_\tau(x,y))d\tau\nonumber\\
&&+\int_0^s(S_{t-\tau}-S_{s-\tau})F(X_\tau^\epsilon(x,y),Y^\epsilon_\tau(x,y))d\tau\nonumber\\
&&+\int_s^tS_{t-\tau}dW^1_\tau+\int_0^s(S_{t-\tau}-S_{s-\tau})dW^1_\tau.\label{Xt-Xs-0}
\end{eqnarray}
In the next, we estimate separately the different terms in
\eqref{Xt-Xs-0}. By using  \eqref{propo-2}, for the first term  we
have
\begin{eqnarray}
 \|(S_t-S_s)x\|\leq C_r\frac{|t-s|^{1-r}} {s^{1-r}}\|x\|.\label{Xt-Xs-1}
 \end{eqnarray}
For the second term, by Lemma \ref{moment-bound} and H\"{o}lder's
inequality, we obtain
\begin{eqnarray}
&&\mathbb{E}\|\int_s^tS_{t-\tau}F(X_\tau^\epsilon(x,y),Y^\epsilon_\tau(x,y))d\tau\|^2\nonumber\\
&&\leq
|t-s|\int_s^t\mathbb{E}\|S_{t-\tau}F(X_\tau^\epsilon(x,y),Y^\epsilon_\tau(x,y))\|^2d\tau\nonumber\\
&&\leq C|t-s|\int_s^t\mathbb{E}
(1+\|X_\tau^\epsilon(x,y)\|^2+\|Y^\epsilon_\tau(x,y)\|^2)d\tau\nonumber\\
&&\leq C_T|t-s|(1+\|x\|^2+\|y\|^2).\label{Xt-Xs-2}
\end{eqnarray}
Concerning the third term, by using \eqref{propo-2}, we can deduce
that
\begin{eqnarray*}
&&\mathbb{E}\|\int_0^s(S_{t-\tau}-S_{s-\tau})F(X_\tau^\epsilon(x,y),Y^\epsilon_\tau(x,y))d\tau\|^2\\
&&\leq\mathbb{E}\left[\int_0^s\|(S_{t-\tau}-S_{s-\tau})F(X_\tau^\epsilon(x,y),Y^\epsilon_\tau(x,y))\|d\tau\right]^2\\
&&\leq C_r\mathbb{E}\left[\int_0^s\frac{(t-s)^{1-r}}{(s-\tau)^{1-r}}
{e^{-\frac{\alpha_1}{2}(s-\tau)}}\|F(X_\tau^\epsilon(x,y),Y^\epsilon_\tau(x,y))\|d\tau\right]^2.\\
\end{eqnarray*}
In view of Lemma \ref{moment-bound} and Minkowski inequality, we get
\begin{eqnarray}
&&\mathbb{E}\|\int_0^s(S_{t-\tau}-S_{s-\tau})F(X_\tau^\epsilon(x,y),Y^\epsilon_\tau(x,y))d\tau\|^2\nonumber\\
&&\leq
C_r|t-s|^{2(1-r)}\left[\int_0^s\frac{e^{-\frac{\alpha_1}{2}(s-\tau)}}{(s-\tau)^{1-r}}
\big(\mathbb{E}\|F(X_\tau^\epsilon(x,y),Y^\epsilon_\tau(x,y))\|^2\big)^{\frac{1}{2}}d\tau\right]^2\nonumber\\
&&\leq
C_r|t-s|^{2(1-r)}\left[\int_0^s\frac{e^{-\frac{\alpha_1}{2}(s-\tau)}}{(s-\tau)^{1-r}}
\mathbb{E}\big(1+\| X_\tau^\epsilon(x,y)\|
+\|Y^\epsilon_\tau(x,y)\|\big)
d\tau\right]^2\nonumber\\
&&\leq C_{r,T}|t-s|^{2(1-r)}(1+\|x\|^2+\|y\|^2),\label{Xt-Xs-3}
\end{eqnarray}
here we have used fact
$\int_0^{+\infty}\frac{e^{-\frac{\alpha_1}{2}s}}{s^{1-r}}ds<+\infty$
in the last step. For the forth term, we directly have
\begin{eqnarray}
\mathbb{E}\|\int_s^tS_{t-\tau}dW^1_\tau\|^2&=&\sum\limits_{k=1}^\infty\lambda_{1,k}\int_s^t\|S_{t-\tau}e_k\|^2d\tau\nonumber\\
&\leq&Tr(Q_1)|t-s|.\label{Xt-Xs-4}
\end{eqnarray}
The final term on the right-hand side of the  \eqref{Xt-Xs-0} can be
treated as follows:
\begin{eqnarray*}
\mathbb{E}\|\int_0^s(S_{t-\tau}-S_{s-\tau})dW^1_\tau\|^2&=&\sum\limits_{k=1}^\infty\lambda_{1,k}\int_0^s\|(S_{t-\tau}-S_{s-\tau})e_k\|^2d\tau\\
&=&\sum\limits_{k=1}^\infty\lambda_{1,k}\int_0^s\|\int_{s-\tau}^{t-\tau}AS_{\rho}e_kd\rho\|^2d\tau\\
&\leq&
C\sum\limits_{k=1}^\infty\lambda_{1,k}\int_0^s|\int_{s-\tau}^{t-\tau}\frac{1}{\rho}d\rho|^2d\tau,
\end{eqnarray*}
here the last inequality following from fact  $\|AS_t\|_{\mathcal
{L}(H)}\leq Ct^{-1}$ for $t>0$. Note that for any $r\in (0,
\frac{1}{2})$, it holds
\begin{eqnarray*}
\int_0^s\left|\int_{s-\tau}^{t-\tau}\frac{1}{\rho}d\rho\right|^2d\tau&\leq&
\int_0^s(s-\tau)^{-2r}\left|\int_{s-\tau}^{t-\tau}\frac{1}{\rho^{1-r}}d\rho\right|^2d\tau\\
&=& r^{-2}\int_0^s(s-\tau)^{-2r}[(t-\tau)^r-(s-\tau)^r]^2d\tau\\
&\leq& r^{-2}\int_0^s(s-\tau)^{-2r}(t-s)^{2r}d\tau\\
&\leq&r^{-2}|t-s|^{2r}\frac{1}{1-2r}s^{1-2r}\\
&\leq&C_{r}|t-s|^{2r}T^{1-2r},
\end{eqnarray*}
which implies that
\begin{eqnarray}
\mathbb{E}\|\int_0^s(S_{t-\tau}-S_{s-\tau})dW^1_\tau\|^2\leq
C_{r,T}|t-s|^{ 2r}. \label{Xt-Xs-5}
\end{eqnarray}
By taking  \eqref{Xt-Xs-1}-\eqref{Xt-Xs-5} into account, we can
deduce that\begin{eqnarray*}
&&(\mathbb{E}\|X_t^\epsilon(x,y)-X^\epsilon_s(x,y)\|^2)^{\frac{1}{2}}\\
&&\leq
C_{r,T}\frac{|t-s|^{1-r}} {s^{1-r}}\|x\|\\
&&+ C_{r,T}(|t-s|^{\frac{1}{2}}+|t-s|^{1-r}+|t-s|^r)(1+\|x\|+\|y\|).
\end{eqnarray*}
\end{proof}
\end{lemma}

\begin{lemma}\label{Yt-Ys}
Let the conditions  (H.1)  and  (H.2) be satisfied and fix $x,y\in
H$ and $T>0$. Then for any $r\in (0,\frac{1}{4})$ there exists a
constant $C_{r, T}>0$ such that for any   $0< s\leq t\leq T$, we
have
\begin{eqnarray*}
(\mathbb{E}\|Y_t^\epsilon(x,y)-Y_s^\epsilon(x,y)\|^2)^{\frac{1}{2}}\leq
C_{r,T}(1+\|x\|+\|y\|)\left[\frac{|t-s|^{r}}{s^{r}}+\frac{|t-s|^{r}}{\epsilon^{r}}\right].
\end{eqnarray*}
\begin{proof}
 We have the decomposition
\begin{eqnarray}
Y_t^\epsilon(x,y)-Y_s^\epsilon(x,y)&=&[S_{\frac{t}{\epsilon}}y-S_{\frac{s}{\epsilon}}y]+
\frac{1}{\epsilon}\int_s^tS_{\frac{t-\tau}{\epsilon}}G(X_\tau^\epsilon(x,y),Y^\epsilon_\tau(x,y))d\tau\nonumber\\
&&+\frac{1}{\epsilon}\int_0^s(S_{\frac{t-\tau}{\epsilon}}-S_{\frac{s-\tau}{\epsilon}})G(X_\tau^\epsilon(x,y),Y^\epsilon_\tau(x,y))d\tau\nonumber\\
&&+\frac{1}{\sqrt{\epsilon}}\int_s^tS_{\frac{t-\tau}{\epsilon}}
dW_\tau^2+\frac{1}{\sqrt{\epsilon}}\int_0^s(S_{\frac{t-\tau}{\epsilon}}-S_{\frac{s-\tau}{\epsilon}})
dW_\tau^2\nonumber\\
&:=&\sum\limits_{k=1}^5J^\epsilon_k(t,s).\label{Yt-Ys-0}
\end{eqnarray}
By \eqref{propo-2}, it is immediate to check that
\begin{eqnarray}
\|J^\epsilon_1(t,s)\|\leq
C_r\frac{|t-s|^r}{s^r}\|y\|.\label{Yt-Ys-1}
\end{eqnarray}
By Minkowski inequality and Lemma \ref{moment-bound}, one can
estimate $J^\epsilon_2(t,s)$ as follows:
\begin{eqnarray}
\mathbb{E}\|J^\epsilon_2(t,s)\|^2&\leq&
\mathbb{E}\left(\frac{C}{\epsilon}\int_s^te^{-\alpha_1\frac{t-\tau}{\epsilon}}(1+
\|X_\tau^\epsilon(x,y)\|
+ \|Y^\epsilon_\tau(x,y)\|)d\tau\right)^2\nonumber\\
&\leq& C\mathbb{E}\left(\int_0^\frac{t-s}{\epsilon}e^{-\alpha_1\tau}
(1+\|X_{t-\epsilon\tau}^\epsilon\|+\|Y_{t-\epsilon\tau}^\epsilon\|)d\tau\right)^2\nonumber\\
&\leq& C\left(\int_0^\frac{t-s}{\epsilon}e^{-\alpha_1\tau}
\Big(\mathbb{E}(1+\|X_{t-\epsilon\tau}^\epsilon(x,y)\|+\|Y_{t-\epsilon\tau}^\epsilon(x,y)\|)^2\Big)^\frac{1}{2}d\tau\right)^2\nonumber\\
&=&C_T(1+\|x\|^2+\|y\|^2)(1-e^{-\alpha_1\frac{t-s}{\epsilon}})^2\nonumber\\
&\leq&C_{r, T}(1+\|x\|^2+\|y\|^2)\frac{|t-s|^{2r}}{\epsilon^{2r}},
\label{Yt-Ys-2}
\end{eqnarray}
where, the last step is due to the inequality $1-e^{-a}\leq C_ra^r$
for $a>0, r\in (0, 1).$ Concerning $J^\epsilon_3(t,s)$,  according
to \eqref{propo-2}, Lemma \ref{moment-bound}  and  Minkowski
inequality, we get
\begin{eqnarray}
&&\!\!\!\!\!\!\!\!\!\mathbb{E}\|J^\epsilon_3(t,s)\|^2\nonumber\\
&\leq&\mathbb{E}\left(\frac{1}{\epsilon}\int_0^s\|S_{\frac{t-\tau}{\epsilon}}-S_{\frac{s-\tau}{\epsilon}}\|_{\mathcal{L}(H)}
 \|G(X_\tau^\epsilon(x,y),Y^\epsilon_\tau(x,y))\|d\tau\right)^2\nonumber\\
&\leq&\mathbb{E}\left(\frac{C_r}{\epsilon}\int_0^s\frac{(t-s)^r}{(s-\tau)^r}e^\frac{-\alpha_1(s-\tau)}{2\epsilon}
(1+ \|X_\tau^\epsilon(x,y)\|+ \|Y_\tau^\epsilon(x,y)\|)d\tau\right)^2\nonumber\\
&\leq&   C_r|t-s|^{2r}  \left(\frac{1}{\epsilon}\int_0^s
\frac{e^\frac{-\alpha_1(s-\tau)}{2\epsilon}}{(s-\tau)^r}
\big(\mathbb{E}(1+ \|X_\tau^\epsilon(x,y)\|+ \|Y_\tau^\epsilon(x,y)\|)^2 \big)^{\frac{1}{2}} d\tau       \right)^2 \nonumber\\
&\leq&C_r|t-s|^{2r}\frac{1}{\epsilon^{2r}}(1+\|x\|^2+\|y\|^2)\left(\int_0^{s/\epsilon}\frac{e^{-\frac{\alpha_1}{2}\tau}}{\tau^r}
d\tau\right)^2\nonumber\\
&\leq&C_r\frac{|t-s|^{2r}}{\epsilon^{2r}}(1+\|x\|^2+\|y\|^2)\left(\int_0^{+\infty}\frac{1}{\tau^r}e^{-\frac{\alpha_1}{2}\tau}d\tau\right)^2\nonumber\\
&\leq&C_r\frac{|t-s|^{2r}}{\epsilon^{2r}}(1+\|x\|^2+\|y\|^2).
\label{Yt-Ys-3}
\end{eqnarray}
For $J^\epsilon_4(t,s)$ we have
\begin{eqnarray}
\mathbb{E}\|J^\epsilon_4(t,s)\|^2&=&\frac{1}{\epsilon}\int_s^t\sum\limits_{k\in\mathbb{N}}
e^{-2(t-\tau)\alpha_k/\epsilon}
d\tau\nonumber\\
&=&\sum\limits_{k\in\mathbb{N}}\int_0^{(t-s)/\epsilon}e^{-2
\tau\alpha_k} d\tau\nonumber\\
&=&\sum\limits_{k\in\mathbb{N}}\frac{1}{2\alpha_k}(1-e^{-2\alpha_k(t-s)/\epsilon})\nonumber\\
&\leq& C_r\frac{|t-s|^{2r}}{\epsilon^{2r}}
\sum\limits_{k\in\mathbb{N}}\frac{1}{\alpha_k^{1-2r}}.\nonumber
\end{eqnarray}
Recalling that we have assumed $r\in (0, \frac{1}{4})$,  it follows
that
$\sum\limits_{k\in\mathbb{N}}\frac{1}{\alpha_k^{1-2r}}<+\infty.$
Therefore, we obtain
\begin{eqnarray}
\mathbb{E}\|J^\epsilon_4(t,s)\|^2&\leq&
C_r\frac{|t-s|^{2r}}{\epsilon^{2r}}.\label{Yt-Ys-4}
\end{eqnarray}
For $J^\epsilon_5(t,s)$ we have
\begin{eqnarray}
\mathbb{E}\|J^\epsilon_5(t,s)\|&=&\frac{1}{\epsilon}\int_0^s\sum\limits_{k\in\mathbb{N}}e^{-2(s-\tau)\alpha_k/\epsilon}
(1-e^{-(t-s)\alpha_k/\epsilon})^2d\tau\nonumber\\
&\leq&\sum\limits_{k\in\mathbb{N}}(1-e^{-(t-s)\alpha_k/\epsilon})^2\frac{1}{2\alpha_k}(1-e^{-2s\alpha_k/\epsilon})\nonumber\\
&\leq&\sum\limits_{k\in\mathbb{N}}(1-e^{-(t-s)\alpha_k/\epsilon})^2\frac{1}{2\alpha_k}\nonumber\\
&\leq&C_r\sum\limits_{k\in\mathbb{N}}\frac{|t-s|^{2r}}{\epsilon^{2r}}\frac{1}{\alpha_k^{1-2r}}\nonumber\\
&\leq&C_r \frac{|t-s|^{2r}}{\epsilon^{2r}}. \label{Yt-Ys-5}
\end{eqnarray}
Collecting together  \eqref{Yt-Ys-1}-\eqref{Yt-Ys-5}, we obtain

\begin{eqnarray*}
(\mathbb{E}\|Y_t^\epsilon(x,y)-Y_s^\epsilon(x,y)\|^2)^{\frac{1}{2}}\leq
C_{r,T}(1+\|x\|+\|y\|)\left[\frac{|t-s|^{r}}{s^{r}}+\frac{|t-s|^{r}}{\epsilon^{r}}\right].
\end{eqnarray*}
\end{proof}
\end{lemma}

\begin{lemma}\label{A-X}
Assume that $x\in \mathscr{D}((-A)^\theta)$ for some $\theta\in (0,
1]$. Then, under  conditions (H.1) and (H.2), we have that
$X_t^\epsilon(x,y)\in \mathscr{D}(-A)$, $\mathbb{P}-a.s.,$ for any
$t>0$ and $\epsilon>0$. Moreover, for any $r\in (0, \frac{1}{4})$ it
holds that
\begin{eqnarray*}
(\mathbb{E}\|AX_t^\epsilon(x,y)\|^2)^\frac{1}{2}\leq C_{r,T}
t^{\theta-1}\|x\|_{(-A)^\theta}+C_{r,T}(1+\|x\|+\|y\|)
(1+\frac{1}{\epsilon^{r}}),\;t\in [0, T].
\end{eqnarray*}

\begin{proof}
For any $t\in [0, T]$ we write $X_t^\epsilon(x,y)$ as
\begin{eqnarray*}
X_t^\epsilon(x,y)&=&\big[S_tx+\int_0^tS_{t-s}F(X_t^\epsilon(x,y),Y_t^\epsilon(x,y))ds\big]\\
&+&\int_0^tS_{t-s}[F(X_s^\epsilon(x,y),Y_s^\epsilon(x,y))-F(X_t^\epsilon(x,y),Y_t^\epsilon(x,y))]ds\\
&+&\int_0^tS_{t-s}dW_s^1\\
&:=&X^{\epsilon,1}_{t}(x,y)+X^{\epsilon,2}_{t}(x,y)+X^{\epsilon,3}_{t}(x,y).
\end{eqnarray*}
For $X^{\epsilon,1}_{t}(x,y)$ we have
\begin{eqnarray*}
\|AX^{\epsilon,1}_{t}(x,y)\|&=&\|AS_tx\|+\|(S_t-I)F(X_t^\epsilon(x,y),Y_t^\epsilon(x,y))\|\\
&\leq&
Ct^{\theta-1}\|x\|_{(-A)^\theta}+C(1+\|X_t^\epsilon(x,y)\|+\|Y_t^\epsilon(x,y)\|),
\end{eqnarray*}
so that, thanks to Lemma \ref{moment-bound}, we obtain
\begin{eqnarray}
(\mathbb{E}\|AX^{\epsilon,1}_{t}(x,y)\|^2)^\frac{1}{2}&=&C_Tt^{\theta-1}\|x\|_{(-A)^\theta}+C_T(1+\|x\|+\|y\|).\label{A-X-1}
\end{eqnarray}
From \eqref{propo-1}, we have
\begin{eqnarray*}
&&\!\!\!\!\!\!\!\!\!\!\|AX^{\epsilon,2}_{t}(x,y)\|\\
&\leq& C\int_0^t\frac{e^{-\frac{\alpha_1}{2}(t-s)}
}{t-s}\|F(X_s^\epsilon(x,y),Y_s^\epsilon(x,y))-F(X_t^\epsilon(x,y),Y_t^\epsilon(x,y))\|ds\\
&\leq&C\int_0^t\frac{e^{-\frac{\alpha_1}{2}(t-s)} }{t-s}[
\|X_t^\epsilon(x,y)-X_s^\epsilon(x,y)\|+
\|Y_t^\epsilon(x,y),Y_s^\epsilon(x,y)\|]ds,
\end{eqnarray*}
which implies
\begin{eqnarray}
\mathbb{E}\|AX^{\epsilon,2}_{t}(x,y)\|^2&\leq&
C\Big[\int_0^t\frac{e^{-\frac{\alpha_1}{2}(t-s)} }{t-s}
(\mathbb{E}\|X_t^\epsilon(x,y)-X_s^\epsilon(x,y)\|^2)^{\frac{1}{2}}
ds\Big]^2\nonumber\\
&+&C\Big[\int_0^t\frac{e^{-\frac{\alpha_1}{2}(t-s)} }{t-s}
(\mathbb{E}\|Y_t^\epsilon(x,y)-Y_s^\epsilon(x,y)\|^2)^{\frac{1}{2}}
ds\Big]^2,\nonumber
\end{eqnarray}
by making use of   Minkowski inequality. If we take  $r\in (0,
\frac{1}{4})$ as in Lemma \ref{Xt-Xs}, we get
\begin{eqnarray}
&&\Big[\int_0^t\frac{e^{-\frac{\alpha_1}{2}(t-s)} }{t-s}
(\mathbb{E}\|X_t^\epsilon(x,y)-X_s^\epsilon(x,y)\|^2)^{\frac{1}{2}}
ds\Big]^2\nonumber\\
&&\leq C_{r,
T}(1+\|x\|+\|y\|)^2\Big[\int_0^{t}\frac{e^{-\frac{\alpha_1}{2}(t-s)}
}{(t-s)^rs^{1-r}}ds+\int_0^{t}\frac{e^{-\frac{\alpha_1}{2}(t-s)}
}{(t-s)^{\frac{1}{2}}
}ds\nonumber\\
&&\qquad+\int_0^{t}\frac{e^{-\frac{\alpha_1}{2}(t-s)} }{(t-s)^{r}
}ds+\int_0^{t}\frac{e^{-\frac{\alpha_1}{2}(t-s)}
}{(t-s)^{1-r }}ds\Big]^2\nonumber\\
&&\leq C_{r,
T}(1+\|x\|+\|y\|)^2\Big[\int_0^{t/2}\frac{e^{-\frac{\alpha_1}{2}(t-s)}
}{(t-s)^rs^{1-r}}ds+\int_{t/2}^t\frac{e^{-\frac{\alpha_1}{2}(t-s)}
}{(t-s)^rs^{1-r}}ds\nonumber\\
&&\qquad+\int_0^{+\infty}\frac{e^{-\frac{\alpha_1}{2}s}
}{s^{\frac{1}{2} }
}ds+\int_0^{+\infty}\frac{e^{-\frac{\alpha_1}{2}s} }{s^{r }
}ds+\int_0^{+\infty}\frac{e^{-\frac{\alpha_1}{2}s}
}{s^{1-r } }ds\Big]^2\nonumber\\
&&\leq C_{r, T}(1+\|x\|+\|y\|)^2\Big[1+\int_0^{t/2}\frac{1}
{(t-s)^rs^{1-r}}ds+\int_{t/2}^t\frac{1}
{(t-s)^rs^{1-r}}ds\Big]^2\nonumber\\
&&\leq C_{r,
T}(1+\|x\|+\|y\|)^2\Big[1+\left(\frac{t}{2}\right)^{-r}\int_0^{t/2}\frac{1}
{ s^{1-r}}ds+\left(\frac{t}{2}\right)^{-(1-r)}\int_{t/2}^t\frac{1}
{(t-s)^r}ds\Big]^2\nonumber\\
&&=C_{r, T}(1+\|x\|+\|y\|)^2(1+\frac{1}{r}+\frac{1}{1-r})\nonumber\\
&&\leq C_{r, T}(1+\|x\|+\|y\|)^2. \label{A-X-1-1.6}
\end{eqnarray}
By using a completely analogous way, due to Lemma \ref{Yt-Ys}, it is
possible to show that
\begin{eqnarray*}
&&\Big[\int_0^t\frac{e^{-\frac{\alpha_1}{2}(t-s)} }{t-s}
(\mathbb{E}\|Y_t^\epsilon(x,y)-Y_s^\epsilon(x,y)\|^2)^{\frac{1}{2}}
ds\Big]^2\\
&&\leq C_{r,
T}(1+\|x\|+\|y\|)^2\Big[\int_0^{t}\frac{e^{-\frac{\alpha_1}{2}(t-s)}
}{(t-s)^{1-r}s^{r}}ds+\frac{1}{\epsilon^{r}}\int_0^{t}\frac{e^{-\frac{\alpha_1}{2}(t-s)}
}{(t-s)^{1-r} }ds\Big]^2\\
&&\leq C_{r, T}(1+\|x\|+\|y\|)^2(1+\frac{1}{\epsilon^{2r}}),
\end{eqnarray*}
which, together with \eqref{A-X-1-1.6}, allows us to  get the
estimate
\begin{eqnarray}
\mathbb{E}\|AX^{\epsilon,2}_{t}(x,y)\|^2\leq C_{r,
T}(1+\|x\|+\|y\|)^2(1+\frac{1}{\epsilon^{2r}}).\label{A-X-1-2}
\end{eqnarray}
Thus, it remains to estimate $AX^{\epsilon,3}_{t}(x,y)$. By
straightforward computations and condition \eqref{Tr-AQ},  we get
\begin{eqnarray}
\mathbb{E}\|AX^{\epsilon,3}_{t}(x,y)\|^2
&=&\mathbb{E}\|\sum\limits_{k\in\mathbb{N}}\sqrt{\lambda_{1,k}}\alpha_ke_k\int_0^te^{-\alpha_k(t-s)}dB_{s,k}^{(1)}\|^2\nonumber\\
&=& \sum\limits_{k\in\mathbb{N}}\lambda_{1,k}
\alpha_k^2\int_0^te^{-2\alpha_{1,k}(t-s)}ds\nonumber\\
&\leq&C\sum\limits_{k\in\mathbb{N}}\lambda_{1,k} \alpha_k\nonumber\\
&\leq& C. \nonumber
\end{eqnarray}
This, together with \eqref{A-X-1} and \eqref{A-X-1-2}, yields
\begin{eqnarray*}
(\mathbb{E}\|AX_t^\epsilon(x,y)\|^2)^\frac{1}{2}\leq C_{r, T}
t^{\theta-1}\|x\|_{(-A)^\theta}+C_{r, T}(1+\|x\|+\|y\|)
(1+\frac{1}{\epsilon^{r}}).
\end{eqnarray*}

\end{proof}

\end{lemma}

\begin{lemma}\label{bar-Xt-Xs}
Let the conditions  (H.1) and (H.2) be satisfied and fix $x \in H$
and $T>0$. Then for any $r\in (0,\frac{1}{4})$ there exists a
constant $C_{r, T}>0$ such that for any   $0< s\leq t\leq T$, we
have
\begin{eqnarray*}
 (\mathbb{E}\| \bar{X}_t(x)-\bar{X}_s(x)\|^2)^\frac{1}{2} &\leq& C_{r, T}\frac{|t-s|^{1-r}} {s^{1-r}}\|x\|\\
&+&C_{r, T}(|t-s|^{\frac{1}{2} }+|t-s|^{1-r}+|t-s|^r)(1+\|x\|).
\end{eqnarray*}

\begin{proof}
It holds that
\begin{eqnarray}
\bar{X}_t(x)
-\bar{X}_s(x)&=&S_tx-S_sx+\int_s^tS_{t-\tau}\bar{F}(\bar{X}_\tau(x))d\tau\nonumber\\
&+&\int_0^s[S_{t-\tau}\bar{F}(\bar{X}_\tau(x))-S_{s-\tau}\bar{F}(\bar{X}_\tau(x))]d\tau\nonumber\\
&+&\int_s^tS_{t-\tau}dW^1_\tau+\int_0^s(S_{t-\tau}-S_{s-\tau})dW^1_\tau.\label{bar-Xt-Xs-0}
\end{eqnarray}
According to \eqref{propo-2}, we obtain
\begin{eqnarray*}
 \|(S_t-S_s)x\|\leq C_r\frac{|t-s|^{1-r}} {s^{1-r}}\|x\|.
 \end{eqnarray*}
For the second term on the right-hand side of \eqref{bar-Xt-Xs-0},
by using  \eqref{bar-x-moment}  we have
\begin{eqnarray*}
\mathbb{E}\|\int_s^tS_{t-\tau}\bar{F}(\bar{X}_\tau(x)
)d\tau\|^2&\leq&
|t-s|\int_s^t\mathbb{E}\|S_{t-\tau}\bar{F}(\bar{X}_\tau(x) )\|^2d\tau\\
&\leq&C|t-s|\int_s^t\mathbb{E}
(1+\|\bar{X}_\tau(x)\|^2 d\tau\\
&\leq&C_T|t-s|(1+\|x\|^2 ).
\end{eqnarray*}
Concerning the third term on the right-hand side of
\eqref{bar-Xt-Xs-0}, we deduce
\begin{eqnarray*}
&&\mathbb{E}\|\int_0^s(S_{t-\tau}-S_{s-\tau})\bar{F}(\bar{X}_\tau(x) )d\tau\|^2\\
&&\leq\mathbb{E}\left[\int_0^s\|(S_{t-\tau}-S_{s-\tau})\bar{F}(\bar{X}_\tau (x))\|d\tau\right]^2\\
&&\leq C_r\mathbb{E}\left[\int_0^s\frac{(t-s)^{1-r}}{(s-\tau)^{1-r}}
{e^{-\frac{\alpha_1}{2}(s-\tau)}}\|\bar{F}(\bar{X}_\tau(x) )\|d\tau\right]^2\\
&&\leq
C_r|t-s|^{2(1-r)}\left[\int_0^s\frac{e^{-\frac{\alpha_1}{2}(s-\tau)}}{(s-\tau)^{1-r}}
\big(\mathbb{E}\|\bar{F}(\bar{X}_\tau(x)
)\|^2\big)^{\frac{1}{2}}d\tau\right]^2,
\end{eqnarray*}
and then, by using once more \eqref{bar-x-moment}, this yields
\begin{eqnarray*}
&&\mathbb{E}\|\int_0^s(S_{t-\tau}-S_{s-\tau})\bar{F}(\bar{X}_\tau(x) )d\tau\|^2\\
&&\leq
C_r|t-s|^{2(1-r)}\left[\int_0^s\frac{e^{-\frac{\alpha_1}{2}(s-\tau)}}{(s-\tau)^{1-r}}
\mathbb{E}\big(1+\| \bar{X}_\tau (x)\|  \big)
d\tau\right]^2\\
&&\leq C_{r,T}|t-s|^{2(1-r)}(1+\|x\|^2).
\end{eqnarray*}
By using arguments analogous to those used in  Lemma \ref{Xt-Xs}, we
have
\begin{eqnarray*}
\mathbb{E}\|\int_s^tS_{t-\tau}dW^1_\tau\|^2 \leq Tr(Q_1)|t-s|
\end{eqnarray*}
and
\begin{eqnarray*}
\mathbb{E}\|\int_0^s(S_{t-\tau}-S_{s-\tau})dW^1_\tau\|^2\leq
C_{r,T}|t-s|^{2r}.
\end{eqnarray*}
Therefore, collecting all estimate of terms appearing on the
right-hand side of \eqref{bar-Xt-Xs-0}, we obtain
\begin{eqnarray*}
 (\mathbb{E}\| \bar{X}_t(x)-\bar{X}_s(x)\|^2)^\frac{1}{2} &\leq& C_{r, T}\frac{|t-s|^{1-r}} {s^{1-r}}\|x\|\\
&+&C_{r, T}(|t-s|^{\frac{1}{2} }+|t-s|^{1-r}+|t-s|^r)(1+\|x\|).
\end{eqnarray*}
\end{proof}
\end{lemma}

\begin{lemma}\label{bar-X}
Assume that $x\in \mathscr{D}((-A)^\theta)$ for some $\theta\in (0,
1]$. Then, under  conditions (H.1)  and (H.2), we have that
$\bar{X}_t \in \mathscr{D}((-A))$, $\mathbb{P}-a.s.,$ for any $t\in
[0, T]$ and $\epsilon>0$. Moreover, it holds that
\begin{eqnarray*}
(\mathbb{E}\|A\bar{X}_t(x) \|^2)^\frac{1}{2}\leq
C_Tt^{\theta-1}\|x\|_{(-A)^\theta}+C_T(1+\|x\|),\;t\in[0, T].
\end{eqnarray*}
\begin{proof}
The proof is analogous to that of the previous Lemma \ref{A-X}. We
write
\begin{eqnarray}
\bar{X}_t (x)&=&\big[S_tx+\int_0^tS_{t-s}\bar{F}(\bar{X}_t(x) )ds\big]\nonumber\\
&&+\int_0^tS_{t-s}[\bar{F}(\bar{X}_s (x))-\bar{F}(\bar{X}_t^\epsilon(x) )]ds+\int_0^tS_{t-s}dW_s^1\nonumber\\
&:=&\bar{X}^{(1)}_{t}(x)+\bar{X}^{(2)}_{t}(x)+\bar{X}^{(3)}_{t}(x).\label{bar-X-0}
\end{eqnarray}
For $\bar{X}^{(1)}_{t}(x)$,  we have
\begin{eqnarray*}
\|A\bar{X}^{(1)}_{t}(x)\|&=&\|AS_tx\|+\|(S_t-I)\bar{F}(\bar{X}_t(x) )\|\\
&\leq& Ct^{\theta-1}\|x\|_{(-A)^\theta}+C(1+\|\bar{X}_t(x) \| ),
\end{eqnarray*}
and then,  thanks to \eqref{bar-x-moment}, we get
\begin{eqnarray*}
(\mathbb{E}\|A\bar{X}^{(1)}_{t}(x)\|^2)^\frac{1}{2}&=&C_Tt^{\theta-1}\|x\|_{(-A)^\theta}+C_T(1+\|x\|).
\end{eqnarray*}
Concerning $\bar{X}^{(2)}_{t}(x)$, we have
\begin{eqnarray*}
\|A\bar{X}^{(2)}_{t}(x)\|&\leq&
C\int_0^t\frac{e^{-\frac{\alpha_1}{2}(t-s)}
}{t-s}\|\bar{F}(\bar{X}_s(x))-\bar{F}(\bar{X}_t(x))\|ds\\
&\leq&C\int_0^t\frac{e^{-\frac{\alpha_1}{2}(t-s)} }{t-s}
\|\bar{X}_t(x)-\bar{X}_s(x)\|ds,
\end{eqnarray*}
and then, according to  Minkowski inequality and Lemma
\ref{bar-Xt-Xs}, for a fixed  $r_0\in (0, \frac{1}{4})$ we obtain
\begin{eqnarray*}
\mathbb{E}\|A\bar{X}^{(2)}_{t}(x)\|^2&\leq&
C\Big[\int_0^t\frac{e^{-\frac{\alpha_1}{2}(t-s)} }{t-s}
(\mathbb{E}\|\bar{X}_t(x) -\bar{X}_s(x) \|^2)^{\frac{1}{2}} ds\Big]^2\\
&&\leq C_{r_0, T}(1+\|x\|
)^2\Big[\int_0^{t}\frac{e^{-\frac{\alpha_1}{2}(t-s)}
}{(t-s)^{r_0}s^{1-r_0}}ds+\int_0^{t}\frac{e^{-\frac{\alpha_1}{2}(t-s)}
}{(t-s)^{r_0}
}ds\\
&& +\int_0^{t}\frac{e^{-\frac{\alpha_1}{2}(t-s)} }{(t-s)^{
\frac{1}{2}} }ds+ +\int_0^{t}\frac{e^{-\frac{\alpha_1}{2}(t-s)}
}{(t-s)^{ 1-r_0} }ds  \Big]^2\\
&\leq&C_{r_0, T}(1+\|x\| )^2.
\end{eqnarray*}
On the other hand, as shown in the Lemma \ref{A-X}, we have
\begin{eqnarray*}
\mathbb{E}\|A\bar{X}^{(3)}_{t}(x)\|^2\leq C.
\end{eqnarray*}
Therefore, collecting all estimates of terms appearing on the
right-hand side of \eqref{bar-X-0}, we can conclude the proof.
\end{proof}
\end{lemma}

\section{Asymptotic expansions} \label{asym}

One of the main tools that we are using in order to prove the main
result  is   It\^{o}'s formula. On the other hand, here the operator
$A$ is unbounded, and then we can not apply directly  It\^{o}'s
formula. Therefore we have to proceed by Galerkin approximation
procedure, to this purpose we need to introduce some notations. For
arbitrary $n\in \mathbb{N}$, let $H^{(n)}$ denote the finite
dimensional subspace of $H$, generated by the set of eigenvectors
$\{e_1,e_2,\cdots, e_n\}$. Let $P_n : H \rightarrow H^{(n)}$ denote
the orthogonal projection defined by
\begin{eqnarray*}
P_nh=\sum\limits_{k=1}^n\Big(h, e_k\Big)_H  e_k,\; h\in H.
\end{eqnarray*}
We define $A_n:H^{(n)}\rightarrow H^{(n)}$ by
\begin{eqnarray*}
A_nh=AP_nh=P_nAh=\sum\limits_{k=1}^n(-\alpha_k)\Big(h,
e_k\Big)_He_k,\;h\in H^{(n)},
\end{eqnarray*}
which is the generator of a strongly semigroup $\{S_{t,n}\}_{t\geq
0}$ on $H^{(n)}$ taking the form
\begin{eqnarray*}
S_{t,n}h=\sum\limits_{k=1}^n e^{-\alpha_kt}\Big(e_k, h\Big)_He_k.
\end{eqnarray*}
Similarly, for arbitrary $n\in \mathbb{N}$ and $\gamma\in
\mathbb{R}$, one can define the $(-A_n)^\gamma:H^{(n)}\rightarrow
H^{(n)}$ as
$$(-A_n)^\gamma h:=\sum\limits_{k=1}^n\alpha_k^{\gamma}\Big(e_k,
h\Big)_He_k,\;h\in H^{(n)}.$$ For each $n$ we consider the
approximating problem of \eqref{abstr-equation}:
\begin{eqnarray}
&&dX_t^{\epsilon,n}=A_nX^{\epsilon,n}_tdt+F_n(X^{\epsilon,n}_t,Y^{\epsilon,n}_t)dt+\sigma_1P_ndW_t^1,\label{abstr-slow-equation-finite}\\
&&dY_t^{\epsilon,n}=\frac{1}{\epsilon}A_nY^{\epsilon,n}_tdt+
\frac{1}{\epsilon}G_n(X^{\epsilon,n}_t,Y^{\epsilon,n}_t)dt+\frac{\sigma_2}{\sqrt{\epsilon}}P_ndW_t^2,\label{abstr-fast-equation-finite}
\end{eqnarray}
with initial conditions $\;X_0^{\epsilon,n}:=x^{(n)}=P_nx,
Y_0^{\epsilon,n}:=y^{(n)}=P_ny$, where $F_n$ and $G_n$ are
respectively defined by
\begin{eqnarray*}
&&F_n(u,v)=P_nF(u,v),\;u,v \in H^{(n)},\\
&&G_n(u,v)=P_nG(u,v),\;\;u,v \in H^{(n)}.
\end{eqnarray*}
Such a problem   is the finite dimensional problem with Lipschitz
coefficients. Under the assumption (H.1) and (H.2), it is easy to
show that the problem
\eqref{abstr-slow-equation-finite}-\eqref{abstr-fast-equation-finite}
admits a unique  \emph{strong  solution} taking values in
$H^{(n)}\times H^{(n)}$, which is denoted by
$(X_t^{\epsilon,n}(x^{(n)},y^{(n)}),Y_t^{\epsilon,n}(x^{(n)},y^{(n)}))$.
Moreover, for any fixed $\epsilon>0$ and $x, y\in H$ it holds that
\begin{eqnarray}
\lim\limits_{n\rightarrow
+\infty}\mathbb{E}\big(\|X_t^{\epsilon}(x,y)-X_t^{\epsilon,n}(x^{(n)},y^{(n)})\|^2
\big)=0\label{approx-infinite-1}
\end{eqnarray}
and
\begin{eqnarray*}
\lim\limits_{n\rightarrow +\infty}\mathbb{E}\big(
\|Y_t^{\epsilon}(x,y)-Y_t^{\epsilon,n}(x^{(n)},y^{(n)})\|^2\big)=0,\label{approx-infinite-2}
\end{eqnarray*}
 For any fixed $x\in H$, we consider frozen  problem associate
with equation \eqref{abstr-fast-equation-finite} in form
 \begin{eqnarray}
dY^{x,n}_t=A_nY^{x,n}_tdt+G_n(P_nx,Y^{x,n}_t)dt+\sigma_2P_ndW^2_t,\;Y^{x,n}_0=y^{(n)}.
\label{frozen-finite}
\end{eqnarray}
Under (H.1) and (H.2), it is easy to check that such a problem
admits a unique strong solution denoted by $Y^{x,n}_t(y^{(n)})$,
which has a unique invariant measure $\mu^{x,n}$ on finite
dimensional space $H^{(n)}$. The averaged equation for finite
dimensional approximation problem \eqref{abstr-slow-equation-finite}
can be defined as follows:
\begin{eqnarray}
d\bar{X}^n_t(x^{(n)})=A_n\bar{X}^n_t(x^{(n)})dt+\bar{F}_n(\bar{X}^n_t(x^{(n)})dt+\sigma_1P_ndW^1_t,\;\bar{X}^n_0=x^{(n)},
\label{aver-finite-equation}
\end{eqnarray}
with
\begin{eqnarray*}
\bar{F}_n(u)=\int_{H^{(n)}}F_n(u,v)\mu^{x,n}(dv),\;u\in H^{(n)}.
\end{eqnarray*}
The averaging principle guarantees
\begin{eqnarray}
\lim\limits_{\epsilon\rightarrow0+}\left\{\mathbb{E}\|{X}^{\epsilon,n}_t(x^{(n)},y^{(n)})
-\bar{X}^n_t(x^{(n)})\|^2\right\}^\frac{1}{2}=0,\label{aver-finite}
\end{eqnarray}
and the above limit is uniform with respect to $n\in \mathbb{N}$. By
triangle inequality we obtain
\begin{eqnarray*}
\mathbb{E}\|\bar{X}_t(x)-\bar{X}^n_t(x^{(n)})\|&\leq&
\mathbb{E}\|\bar{X}_t(x)-{X}^{\epsilon
}_t(x,y)\|\\
&+&\mathbb{E}\|{X}^{\epsilon}_t(x,y)-{X}^{\epsilon,n}_t(x^{(n)},y^{(n)})\|\\
&+&\mathbb{E}\|{X}^{\epsilon,n}_t(x^{(n)},y^{(n)})-\bar{X}^n_t(x^{(n)})\|,
\end{eqnarray*}
which, together with \eqref{aver-infinite} and \eqref{aver-finite},
yields
\begin{eqnarray}
\lim\limits_{n\rightarrow\infty}\mathbb{E}\|\bar{X}_t(x)-\bar{X}^n_t(x^{(n)})\|=0.\label{approx-infinite-averaging}
\end{eqnarray}

\begin{remark}
Note that for any $T>0$  and $\phi\in C_b^3(H, \mathbb{R})$ we have
\begin{eqnarray*}
\left|\mathbb{E}\phi(X^\epsilon_T(x,y))-\mathbb{E}\phi(\bar{X}_T(x))\right|
&\leq&\left|\mathbb{E}\phi(X^\epsilon_T(x,y))-\mathbb{E}\phi(X^{\epsilon,n}_T(x^{(n)},y^{(n)} ))\right|\\
&+&\left|\mathbb{E}\phi(X^{\epsilon,n}_T(x^{(n)},y^{(n)}))-\mathbb{E}\phi(\bar{X}^{n}_T(x^{(n)}
))\right|\\
&+&\left|\mathbb{E}\phi(\bar{X}^{n}_T(x^{(n)}
))-\mathbb{E}\phi(\bar{X}_T(x))\right|.
\end{eqnarray*}
According to the approximation results \eqref{approx-infinite-1} and
\eqref{approx-infinite-averaging} the first and last terms above
converge to zero as $n$ goes to infinity. In order to prove Theorem
\ref{theorem} we have only to show that for any  $r\in (0, 1)$, it
holds
\begin{eqnarray}
\left|\mathbb{E}\phi(X^{\epsilon,n}_T(x^{(n)},y^{(n)}))-\mathbb{E}\phi(\bar{X}^{n}_T(x^{(n)}))\right|\leq
C_{\theta, r, T,\phi,x,y}\epsilon^{1-r} \label{finite-weak}
\end{eqnarray}
for some constant $C_{\theta, r, T,\phi,x,y}$ independent of the
dimension index $n$.
\end{remark}
\begin{remark}
For all $n\in\mathbb{N}$, the regular conditions  on drift
coefficients $F$ and $G$ presented in (H.1) and (H.2) are still
valid for $F_n$ and $G_n$, respectively, but replacing $H$  by
$H^{(n)}$. In particular, the boundedness on derivatives associated
with  $F_n$ and $G_n$ are uniform with respect to dimension $n$. As
a result, all properties satisfied by $(X_t^{\epsilon },
Y_t^{\epsilon })$,
 $Y_t^x$ and $P_t^x$ are still valid for $(X_t^{\epsilon,n},
Y_t^{\epsilon,n})$, $Y_t^{x,n}$  and for the transition semigroup
$P_t^{x,n}$ corresponding to \eqref{frozen-finite}, respectively.
Moreover, all estimates for $(X_t^{\epsilon,n}, Y_t^{\epsilon,n})$,
$Y_t^{x,n}$ and $P_t^{x,n}$ are uniform with respect to $n\in
\mathbb{N}$. Similarly, $\bar{F}_n$ and $\bar{X}_t^{n}$ inherit all
properties described for $\bar{F}$ and $\bar{X}_t$, respectively,
with all estimates uniform with respect to $n\in \mathbb{N}$.
\end{remark}

\begin{remark}
{In what follows}, the letter $C$ below with or without subscripts
will denote generic positive constants independent of $\epsilon$ and
dimension $n$,  whose value may  change from one line to another.
\end{remark}

Let $\phi$ be the test function as in Theorem \ref{theorem}. As
usual, we use the notation  $(X_t^{\epsilon,n}(x,y),
Y_t^{\epsilon,n}(x,y))$ to denote the solution to equation
\eqref{abstr-slow-equation-finite}-\eqref{abstr-fast-equation-finite}
with initial value $(X_0^{\epsilon,n}(x,y),
Y_0^{\epsilon,n}(x,y))=(x,y)\in H^{(n)}\times H^{(n)}$. For any
$n\in\mathbb{N}$, we define a function    $u_n^\epsilon: [0,
T]\times H^{(n)} \times H^{(n)}\rightarrow \mathbb{R}$ by
$$u_n^\epsilon(t, x,y)=\mathbb{E}\phi(X_t^{\epsilon,n}(x,y)).$$

 We now introduce  two differential operators associated with the
 fast variable system \eqref{abstr-fast-equation-finite} and slow variable system \eqref{abstr-slow-equation-finite}
 in finite dimensional space, respectively:
\begin{eqnarray*}
\mathcal {L}_1^n\varphi(y)&=&\Big(A_ny+G_n(x,y),
D_y\varphi(y)\Big)_H\\
&&+\frac{1}{2}\sigma_2^2Tr(D^2_{yy}\varphi(y)Q_{2,n}^{\frac{1}{2}}(Q_{2,n}^{\frac{1}{2}})^*),\;
\varphi \in C_b^2(H^{(n)},\mathbb{R})
\end{eqnarray*}
and
\begin{eqnarray*}
\mathcal {L}_2^n\varphi(x)&=&\Big(A_nx+ F_n(x,y),
D_x\varphi(x)\Big)_{ {H}
}\\
&&+\frac{1}{2}\sigma_1^2Tr(D^2_{xx}\varphi(x)Q_{1,n}^{\frac{1}{2}}(Q_{1,n}^{\frac{1}{2}})^*),\;
\varphi \in C_b^2(H^{(n)},\mathbb{R} ),
\end{eqnarray*}
where $Q_{1,n}:=Q_1P_n$ and  $Q_{2,n}:=Q_2P_n$ for any
$n\in\mathbb{N}$. It is known that $u_n^\epsilon$ is a solution to
the forward Kolmogorov equation:
\begin{eqnarray}\label{Kolm}
\begin{cases}
\frac{\partial}{\partial t}u_n^\epsilon(t, x, y)=\mathcal {L}^{\epsilon, n} u_n^\epsilon(t, x, y),\\
u_n^\epsilon(0, x,y)=\phi(x),
\end{cases}
\end{eqnarray}
where $\mathcal {L}^{\epsilon, n}:=\frac{1}{\epsilon}\mathcal
{L}_1^n+\mathcal {L}_2^n.$

 Also recall the Kolmogorov operator for the averaged system \eqref{aver-finite-equation} is
defined as
\begin{eqnarray*}
\bar{\mathcal {L}^n}\varphi(x)&=&\Big(A_nx+\bar{F}_n(x),
D_x\varphi(x)\Big)_{ {H}
}\\
&&+\frac{1}{2}\sigma_1^2Tr(D^2_{xx}\varphi(x)Q_{1,n}^{\frac{1}{2}}(Q_{1,n}^{\frac{1}{2}})^*),\;
\varphi\in C_b^2(H^{(n)}, \mathbb{R} ).
\end{eqnarray*}
If we set $$\bar{u}_n(t, x)=\mathbb{E}\phi(\bar{X}_t^n(x)),$$ we
have
\begin{eqnarray}\label{Kolm-Aver}
\begin{cases}
\frac{\partial }{\partial t}\bar{u}_n(t, x)=\bar{\mathcal {L}^n} \bar{u}_n(t, x),\\
\bar{u}_n(0, x)=\phi(x).
\end{cases}
\end{eqnarray}
Then the weak difference at time $T$ can be rewritten as
\begin{eqnarray*}
\mathbb{E}\phi({X}^{\epsilon,n}_T)-\mathbb{E}\phi(\bar{X}^n_T)=u_n^\epsilon(T,
x,y)-\bar{u}_n(T,x).
\end{eqnarray*}
Henceforth, for the sake of brevity, we will omit to write the
dependence of the temporal variable $t$ and spatial variables $x$
and $y$ in some occasion. For example,  we often write $u^\epsilon$
instead of $u_n^\epsilon(t, x, y)$. Our aim is to seek an expansion
for   $u_n^\epsilon(T, x,y)$ with the form

\begin{eqnarray}\label{asymp-expan}
u_n^\epsilon=u_{0,n}+\epsilon u_{1,n}+r^\epsilon_n,
\end{eqnarray}
where  $u_{0,n}$ and $u_{1,n}$ are smooth functions which will be
constructed below, and $r^\epsilon_n$ is the remainder term.

The rest of this section is devoted to the proof of Theorem
\ref{theorem}. We will proceed in several steps, which have been
structured as subsections.

\subsection{\textbf{The leading term}}
Let us first   determine the leading term.  Now, substituting
expansions \eqref{asymp-expan} into \eqref{Kolm} yields
\begin{eqnarray*}
\frac{\partial u_{0,n}}{\partial t}+\epsilon\frac{\partial
u_{1,n}}{\partial t}+\frac{\partial r^\epsilon_n}{\partial t}&=&
\frac{1}{\epsilon}\mathcal {L}_1^nu_{0,n}+\mathcal {\mathcal
{L}}_1^nu_{1,n}+\frac{1}{\epsilon}\mathcal {L}_1^nr^\epsilon_n\\
&&+\mathcal {L}_2^nu_{0,n}+\epsilon \mathcal {L}_2^nu_{1,n}+\mathcal
{L}_2^nr^\epsilon_n.
\end{eqnarray*}
By comparing orders  of $\epsilon$, we obtain
\begin{eqnarray}
&&\mathcal {L}_1^nu_{0,n}=0  \label{u-o-equ-1}
\end{eqnarray}
and
\begin{eqnarray}
&&\frac{\partial u_{0,n}}{\partial t}=\mathcal
{L}_1^nu_{1,n}+\mathcal {L}_2^nu_{0,n}.\label{u-0-equ-2}
\end{eqnarray}
 It follows from \eqref{u-o-equ-1} that $u_{0,n}$ is independent of
$y$, which means $$u_{0,n}(t,x, y)=u_{0,n}(t,x).$$ We also impose
the initial condition $u_{0,n}(0,x)=\phi(x).$ Since $\mu^{x,n}$ is
the invariant measure of a Markov process with generator $\mathcal
{L}_1^n$, we have
\begin{eqnarray*}
\int_{H^{(n)}}\mathcal {L}_1^nu_{1,n}(t,x,y)\mu^{x,n}(dy)=0,
\end{eqnarray*}
which, by invoking \eqref{u-0-equ-2}, implies
\begin{eqnarray*}
\frac{\partial u_{0,n}}{\partial t}(t,x)&=&\int_{H^{(n)}}\frac{\partial u_{0,n}}{\partial t}(t,x)\mu^{x,n}(dy)\\
&=&\int_{H^{(n)}}\mathcal {L}_2^nu_{0,n}(t,x)\mu^{x,n}(dy)\\
&=&\left(A_nu_{0,n}(t,x)+\int_{H^{(n)}} F_n(x,y)\mu^{x,n}(dy),
D_xu_{0,n}(t,x)\right)_{ {H}}\nonumber\\
&&+\frac{1}{2}\sigma_1^2Tr(D^2_{xx}u_{0,n}(t,x)Q_{1,n}^{\frac{1}{2}}(Q_{1,n}^{\frac{1}{2}})^*)\\
&=&\bar{\mathcal {L}^n}u_{0,n}(t,x),
\end{eqnarray*}
so that $u_{0,n}$ and $\bar{u}_n$ satisfy the same evolution
equation. By using   a uniqueness argument, such $u_{0,n}$ has to
coincide with the solution $\bar{u}_n$  and we have the following
lemma:
\begin{lemma}\label{u-0}
Assume (H.1) and (H.2). Then for any $x, y \in H^{(n)}$  and $T>0$,
we have $u_{0,n}(T,x,y)=\bar{u}_n(T,x)$.
\end{lemma}

\subsection{\textbf{Construction of} $ {u_{1,n}}$}
Let us proceed to carry out the construction of $u_{1,n}$. Thanks to
Lemma \ref{u-0} and \eqref{Kolm-Aver}, the equation
\eqref{u-0-equ-2} can be rewritten as
\begin{eqnarray*}
\bar{\mathcal
{L}}^n\bar{u}_n=\mathcal{L}^n_1u_{1,n}+\mathcal{L}^n_2\bar{u}_n,
\end{eqnarray*}
and hence we get an elliptic equation for $u_{1,n}$ with form
\begin{eqnarray*}
\mathcal {L}_1^nu_{1,n}(t,x,y)=\Big(\bar{ F}_n(x)-F_n(x,y),
D_x\bar{u}_n(t,x)\Big)_{{H} }:=-\rho_n(t,x,y),
\end{eqnarray*}
where $\rho_n$ is of class $\mathcal{C}^2$ with respect to $y$, with
uniformly bounded derivatives. Moreover, it satisfies for any $t\geq
0$ and $x\in H^{(n)}$,
\begin{eqnarray*}
\int_{H^{(n)}}\rho_n(t,x,y)\mu^{x, n}(dy)=0.
\end{eqnarray*}
{For any} $y\in H^{(n)}$ and $s>0$ we have
\begin{eqnarray*}
\frac{\partial}{\partial s}P_{s,n}^x\rho_n(t,x,y)&=&\Big(A_ny+G_n(x,y),D_y(P_{s,n}^x\rho_n(t,x,y))\Big)_H\nonumber\\
&+&\frac{1}{2}\sigma_2^2Tr[D^2_{yy}(P_{s,n}^x\rho_n(t,x,y))Q_{2,n}^{\frac{1}{2}}(Q_{2,n}^{\frac{1}{2}})^*],
\end{eqnarray*}
here $$P_{s,n}^x\rho_n(t,x,y)=\mathbb{E}\rho_n(t, x,Y_s^{x, n}(y))$$
satisfying
\begin{equation}
\lim\limits_{s\rightarrow{+\infty}}\mathbb{E}\rho_n(t,
x,Y_s^{x,n}(y))=\int_{H^{(n)}}\rho_n(t,x,z)\mu^{x,n}(dz)=0.\label{rho-limits}
\end{equation}   Indeed, by
the invariant property of $\mu^{x,n}$ and Lemma \ref{ux} in the next
section,
\begin{eqnarray*}
&&\left|\mathbb{E}\rho_n(t,
x,Y_s^{x,n}(y))-\int_{H^{(n)}}\rho_n(t,x,z)\mu^{x,n}(dz)\right|\nonumber\\
&&=\left|\int_{H^{(n)}}\mathbb{E}[\rho_n(t, x,Y_s^{x,n}(y))-
\rho_n(t,
x,Y_s^{x,n}(z))\mu^{x,n}(dz)]\right|\nonumber\\
&&\leq\int_{H^{(n)}}\left|\mathbb{E}\Big( F_n(x,Y_s^{x,n}(z))-
F_n(x,Y_s^{x,n}(y)), D_x\bar{u}_n(t,x)\Big)_{{H}
}\right|\mu^{x,n}(dz)\nonumber\\
&&\leq C\int_{H^{(n)}} \mathbb{E} \|Y_s^{x,n}(z)- Y_s^{x,n}(y)\|
\mu^{x,n}(dz).\nonumber\\
\end{eqnarray*}
This, in view of \eqref{initial-diff} and \eqref{mu-Momenent-bound},
yields
\begin{eqnarray*}
&&\left|\mathbb{E}\rho_n(t,
x,Y_s^{x,n}(y))-\int_{H^{(n)}}\rho_n(t,x,z)\mu^{x,n}(dz)\right|\nonumber\\
&&\leq Ce^{-\frac{\beta}{2}s}(1+\|x\|+\|y\|),
\end{eqnarray*}
which implies the equality \eqref{rho-limits}. Therefore, we get
\begin{eqnarray}
&&\Big(A_ny+G_n(x,y),D_y\int_0^{{+\infty}}P_{s, n}^x\rho_n(t,x,y)
ds\Big)_H\nonumber\\
&&+\frac{1}{2}\sigma_2^2Tr[D^2_{yy}\int_0^{{+\infty}}(P_{s,n}^x\rho_n(t,x,y))Q_{2, n}^{\frac{1}{2}}(Q_{2,n}^{\frac{1}{2}})^*]ds\nonumber\\
&&=\int_0^{{+\infty}}\frac{\partial}{\partial s}P_{s,n}^x\rho_n(t,x,y)ds\nonumber\\
&&=\lim\limits_{s\rightarrow{+\infty}}\mathbb{E}\rho_n(t,
x,Y_s^{x,n}(y))-\rho_n(t,x,y)\nonumber\\
&&=\int_{H^{(n)}}\rho_n(t,x,y)\mu^{x,n}(dy)-\rho_n(t,x,y)\nonumber\\
&&=-\rho_n(t,x,y),\nonumber
\end{eqnarray}
which means
$\mathcal{L}_1^n(\int_0^{{+\infty}}P_{s,n}^x\rho_n(t,x,y)
ds)=-\rho_n(t,x,y).$ Therefore, we can set
\begin{eqnarray}\label{u-1}
u_{1,n}(t,x,y)=\int_0^{+\infty}
\mathbb{E}\rho_n(t,x,Y^{x,n}_s(y))ds.
\end{eqnarray}
\begin{lemma}\label{u-1-abso-lemma}
Assume (H.1)  and (H.2).  {Then for any }$x, y\in H^{(n)}$ and
$T>0$, we have
\begin{eqnarray}
|u_{1,n}(t,x,y)|\leq C_T(1 +\|x\|+\|y\|),\;t\in[0, T].
\label{u-1-abso}
\end{eqnarray}
\begin{proof}
As known from \eqref{u-1}, we have
\begin{equation*}
u_{1,n}(t,x,y)=\int_0^{{+\infty}}\mathbb{E}\Big(\bar{ F}_n(x)-
F_n(x,Y_s^{x,n}(y)), D_x\bar{u}_n(t,x)\Big)_{ {H} }ds.
\end{equation*}
This implies that
\begin{eqnarray*}
|u_{1,n}(t,x,y)|&\leq&\int_0^{{+\infty}}\|\bar{F}_n(x)- \mathbb{E}[
F_n(x,Y_s^{x,n}(y))]\|\cdot\|D_x\bar{u}_n(t,x)\| ds.
\end{eqnarray*}
Then, in view of Lemma \ref{ux} and \eqref{Averaging-Expectation},
this implies:
\begin{eqnarray*}
 |u_{1,n}(t,x,y)|&\leq& C_T(1+\|x\|+\|y\|)\int_0^{{+\infty}}e^{-\frac{\beta}{2} s}ds\\
&\leq&C_T(1+\|x\|+\|y\|).
\end{eqnarray*}
\end{proof}
\end{lemma}

\subsection{\textbf{Determination of remainder} $ {r^\epsilon_n}$}
Once  $u_{0,n}$ and $u_{1,n}$   have been determined, we can carry
out the construction of the remainder $r^\epsilon_n$. It is known
that
\begin{eqnarray*}
(\partial_t-\mathcal{L}^{\epsilon,n})u^\epsilon_n=0,
\end{eqnarray*}
which, together with \eqref{u-o-equ-1} and \eqref{u-0-equ-2},
implies
\begin{eqnarray*}
(\partial_t-\mathcal{L}^{\epsilon,n})r^\epsilon_n&=&-(\partial_t-\mathcal{L}^{\epsilon,n})
u_{0,n}-\epsilon(\partial_t-\mathcal{L}^{\epsilon,n})u_{1,n}\\
&=&-(\partial_t-\frac{1}{\epsilon}\mathcal{L}_1^n-\mathcal{L}_2^n)u_{0,n}-\epsilon(\partial_t
-\frac{1}{\epsilon}\mathcal{L}_1^n-\mathcal{L}_2^n)u_{1,n}\\
&=&\epsilon(\mathcal{L}_2^nu_{1,n}-\partial_tu_{1,n}).
\end{eqnarray*}
In order to estimate the remainder term $r^\epsilon_n$  we need the
following crucial lemmas.

\begin{lemma}\label{4.3}
 {Assume  that  $x, y\in H^{(n)}$}. Then, under conditions (H.1) and (H.2), for any
$T>0$  and $\theta\in(0, 1]$ we have
\begin{eqnarray*}
\left|\frac{\partial u_{1,n}}{\partial t}(t,x,y)\right|
&\leq&C_T(1+\frac{1}{t}+t^{\theta-1})(1+\|x\|+\|y\|+\|x\|_{(-A_n)^\theta})^2,\;
t\in [0, T].
\end{eqnarray*}

\begin{proof}
According to \eqref{u-1}, we have
\begin{equation}
\frac{\partial u_{1,n}}{\partial
t}(t,x,y)=\int_0^{{+\infty}}\mathbb{E}\left(\bar{F}_n(x)-
{F}_n(x,Y_s^{x,n}(y)), \frac{\partial }{\partial
t}D_x\bar{u}_n(t,x)\right)_{{H}}ds.\label{u1=deriv}
\end{equation}
For any $h\in H^{(n)},$
\begin{eqnarray}
D_x\bar{u}_n(t,x)\cdot h&=&\mathbb{E}[\phi'(\bar{X}_t^n(x))\cdot
D_x\bar{X}_t^n(x)\cdot h]\nonumber\\
&=&\mathbb{E}\Big(\phi'(\bar{X}_t^n(x)),\eta^{h,x,n}_t\Big)_H,\label{u(t,x)h-deriv}
\end{eqnarray}
here $\eta^{h,x,n}_t$ is the mild solution (also \emph{strong
solution}) of variation equation  corresponding to the problem
\eqref{aver-finite-equation} in form
\begin{eqnarray*}
\begin{cases}
d\eta ^{h,x,n}_t=\left(A_n\eta ^{h,x,n}_t+ \bar{{F}}_n'(\bar{X}_t^n(x))\cdot\eta ^{h,x,n}_t\right)dt,\\
\eta ^{h,x,n}_0=h.
\end{cases}
\end{eqnarray*}
Keep in mind that $\bar{X}_t^n(x)$ is the \emph{strong solution} of
equation \eqref{aver-finite-equation} with initial value
$\bar{X}_0^n(x)=x$. By It\^{o}'s formula in finite dimensional
spaces, we get
\begin{eqnarray*}
\phi'(\bar{X}_t^n(x))&=&\phi'(x)+\int_0^t\phi''(\bar{X}_s^n(x))\cdot
[A_n\bar{X}_s^n(x)+\bar{F}_n(\bar{X}_s^n(x))]ds\\
&+&\int_0^t\phi''(\bar{X}_s^n(x))dW_s^{1,n}\\
&+&\frac{1}{2}\sum\limits_{k=1}^n\int_0^t\phi'''(\bar{X}_s^n(x))\cdot\Big(\sqrt{\lambda_{1,k}}e_k,\sqrt{\lambda_{1,k}}e_k\Big)ds,
\end{eqnarray*}
where
$W_t^{1,n}:=\sum\limits_{k=1}^n\sqrt{\lambda_{i,k}}B^{(i)}_{t,k}e_k
$ denotes the $Q^{1,n}-$Wiener process in $H^{(n)}$. Then, by using
again It\^{o}'s formula, after taking the expectation we have
\begin{eqnarray*}
&&\!\!\!\!\!\!\!\!\!\mathbb{E}\Big(\phi'(\bar{X}_t^n(x)),\eta^{h,x,n}_t\Big)_H\\
&=&\Big(\phi'(x),h\Big)_H\\
&+&\mathbb{E}\int_0^t\Big(\eta_s^{h,x,n},\phi''(\bar{X}_s^n(x))\cdot[A_n\bar{X}_s^n(x)+\bar{F}_n(\bar{X}_s^n(x))]\Big)_Hds\\
&+&\mathbb{E}\int_0^t\Big(\phi'(\bar{X}_s^n(x)),A_n\eta_s^{h,x,n}+\bar{F}'_n(\bar{X}_s^n(x))\cdot\eta_s^{h,x,n}\Big)_Hds\\
&+&\frac{1}{2}\mathbb{E}\sum\limits_{k=1}^n\int_0^t\Big(\phi'''(\bar{X}_s^n(x))\cdot
\big(\sqrt{\lambda_{1,k}}e_k,\sqrt{\lambda_{1,k}}e_k\big),\eta_s^{h,x,n}\Big)_Hds.
\end{eqnarray*}
Now, returning to \eqref{u(t,x)h-deriv} and differentiating with
respect to $t$, we obtain
\begin{eqnarray*}
\frac{\partial}{\partial t}(D_x\bar{u}_n(t,x)\cdot h)&=&\mathbb{E} \Big(\eta_t^{h,x,n},\phi''(\bar{X}_t^n(x))\cdot[A_n\bar{X}_t^n(x)+\bar{F}_n(\bar{X}_t^n(x))]\Big)_H\\
&+&\mathbb{E}\Big(\phi'(\bar{X}_t^n(x)),A_n\eta_t^{h,x,n}+\bar{F}'_n(\bar{X}_t^n(x))\cdot\eta_t^{h,x,n}\Big)_H\\
&+&\frac{1}{2}\mathbb{E}\sum\limits_{k=1}^n\Big(\phi'''(\bar{X}_t^n(x))\cdot\big(\sqrt{\lambda_{1,k}}e_k,\sqrt{\lambda_{1,k}}e_k\big),\eta_t^{h,x,n}\Big)_H,
\end{eqnarray*}
so that
\begin{eqnarray*}
\left|\frac{\partial}{\partial t}(D_x\bar{u}_n(t,x)\cdot
h)\right|&\leq&C\mathbb{E} \left[\|\eta_t^{h,x,n}\|
(\|A_n\bar{X}_t^n(x)\|+\|\bar{F}_n(\bar{X}_t^n(x))\|)\right]\\
&+&C\mathbb{E}
\|A_n\eta_t^{h,x,n}\|+C\mathbb{E}\|\bar{F}'_n(\bar{X}_t^n(x))\cdot\eta_t^{h,x,n}\|\\
&+& \mathbb{E}\sum\limits_{k=1}^\infty{\lambda_k} \|\eta_t^{h,x,n}\|.\\
\end{eqnarray*}
Then,  as  \eqref{bar-F-derivative} holds, by using Lemma
\ref{bar-X}, Lemma \ref{eta-bound} and Lemma \ref{eta-regularity},
it follows
\begin{eqnarray*}
\left|\frac{\partial}{\partial t}(D_x\bar{u}_n(t,x)\cdot
h)\right|&\leq&C_T\|h\|(t^{\theta-1}\|x\|_{(-A_n)^\theta}+1+\|x\|)\\
&&+C\|h\|(1+\frac{1}{t})(1+\|x\|)\\
&\leq&C\|h\|(1+\frac{1}{t}+t^{\theta-1}\|x\|_{(-A_n)^\theta}
+\frac{1}{t}\|x\|+\|x\|).
\end{eqnarray*}
Hence, as \eqref{u1=deriv} holds, from the above estimate and
\eqref{Averaging-Expectation} we get
\begin{eqnarray*}
&&\!\!\!\!\!\!\!\!\!\!\!\left|\frac{\partial u_{1,n}}{\partial
t}(t,x,y)\right| \leq
C_T(1+\frac{1}{t}+t^{\theta-1}\|x\|_{(-A_n)^\theta}
+\frac{1}{t}\|x\|+\|x\|)\\
&&\qquad\qquad\qquad\cdot\int_0^{{+\infty}}\mathbb{E}
\|\bar{ F}_n(x)-  {F}_n(x,Y_s^{x,n}(y))\|ds\\
&\leq&C_T(1+\frac{1}{t}+t^{\theta-1}\|x\|_{(-A_n)^\theta}
+\frac{1}{t}\|x\|+\|x\|)(1+\|x\|+\|y\|)\int_0^{+\infty}e^{-\frac{\beta}{2}s}ds\\
&\leq&C_T(1+\frac{1}{t}+t^{\theta-1}\|x\|_{(-A_n)^\theta}
+\frac{1}{t}\|x\|+\|x\|)(1+\|x\|+\|y\|)\\
&\leq&C_T(1+\frac{1}{t}+t^{\theta-1})\left(1+\|x\|+\|y\|+\|x\|_{(-A_n)^\theta}\right)^2.
\end{eqnarray*}
\end{proof}
\end{lemma}

\begin{lemma}\label{4.4}
 {Assume   }$x\in H^{(n)}$ and $y\in H^{(n)}$. Then,
under conditions (H.1) and (H.2), for any $T>0$  we have
\begin{eqnarray*}
\left|\mathcal {L}_2^nu_{1,n}(t,x,y)\right|\leq
C_T\big(1+\|A_nx\|+\|x\|+\|y\| \big)\big(1+\|x\|+\|y\| \big), \;t\in
[0, T].
\end{eqnarray*}
\begin{proof}
As known,  {for any} $x\in H^{(n)}$ it holds
\begin{eqnarray*}
\mathcal {L}_2^nu_{1,n}(t,x,y)&=&\Big(A_nx+  F_n(x, y),
D_xu_{1,n}(t,x,y)
\Big)_{H}\\
&+&\frac{1}{2}\sigma_2^2Tr\Big(D^2_{xx}u_{1,n}(t,x,y)Q_{1,n}^\frac{1}{2}(Q_{1,n}^\frac{1}{2})^*\Big).
\end{eqnarray*}
We will carry out the estimate of $\left|\mathcal {L}_2^nu_{1,n}(t,x,y)\right|$ in two steps.\\
(\textbf{Step 1}) Estimate of $\Big( {A}_nx+ {F}_n(x, y),
D_xu_{1,n}(t,x,y) \Big)_{H}$.

 For any $k\in {H}^{(n)}$, we have
\begin{eqnarray*}
& &D_xu_{1,n}(t,x,y)\cdot k\\
&&=\int_0^{{+\infty}}\Big(D_x(\bar{
F}_n(x)-\mathbb{E}{F}_n(x,Y_s^{x,n}(y)))\cdot k,D_x\bar{u}_n(t,x)\Big)_{{H}}ds\\
&&+\int_0^{{+\infty}}\Big(\bar{F}_n(x)-\mathbb{E} {F}_n(x,Y_s^{x,
n}(y)),D^2_{xx}\bar{u}_n(t,x)\cdot
k\Big)_{ {H} }ds\\
&&:=I_{1,n}(t,x,y,k)+I_{2,n}(t,x,y,k).
\end{eqnarray*}
Directly, we have
\begin{eqnarray}
&&|I_{1,n}(t,x,y,k)|\nonumber\\
&&\leq \int_0^{+\infty} \left| \Big(D_x(\bar{F}_n(x)-
\mathbb{E}{F}_n(x,Y_s^{x,n}(y)))\cdot
k,D_x\bar{u}_n(t,x)\Big)_{{H}}\right| ds.\nonumber
\end{eqnarray}
By making use of \eqref{mix-derivative}, the above yields
\begin{eqnarray}
|I_{1,n}(t,x,y,k)|&\leq&C\|k\| \cdot\|D_x\bar{u}_n(t,x)\|
\cdot\int_0^{+\infty}
e^{-c s} (1+\|x\|+\|y\|)ds\nonumber\\
&\leq& C\|k\| \cdot\|D_x\bar{u}_n(t,x)\|(1+\|x\|+\|y\|)\nonumber \\
&\leq& C\|k\|(1+\|x\|+\|y\|), \label{I_1}
\end{eqnarray}
where we used  Lemma \ref{ux} in the last step. By Lemma \ref{uxx}
and \eqref{Averaging-Expectation},  we have
\begin{eqnarray}
|I_{2,n}(t,x,y,k)|&\leq& \int_0^{{+\infty}}\left|\Big(\bar{F}_n(x)-
\mathbb{E} {F}_n(x,Y_s^{x,n}(y)),D^2_{xx}\bar{u}_n(t,x)\cdot k\Big)_{{H}}\right|ds\nonumber\\
&\leq& C\|k\|\int_0^{{+\infty}}\|\bar{F}_n(x)-
\mathbb{E}{F}_n(x,Y_s^{x,n}(y))\|ds\nonumber\\
&\leq&C\|k\|(1+\|x\|  +\|y\|)\int_0^{+\infty} e^{-\frac{\beta}{2} s} ds\nonumber\\
&\leq&C\|k\|(1 +\|x\| +\|y\|).\nonumber
\end{eqnarray}
Together with \eqref{I_1}, this allows us to get
\begin{eqnarray*}
\left|D_xu_{1,n}(t,x,y)\cdot k\right|&\leq&C\|k\| (1+\|x\|+ \|y\|)
\end{eqnarray*}
which means
\begin{eqnarray}
&&\left|\Big({A}_nx+ {F}_n(x,
y), D_xu_{1,n}(t,x,y) \Big)_{H}\right|\nonumber\\
&&\leq C\big(1+\|{A}_nx\|+\|x\|+\|y\| \big)\big(1+\|x\|+\|y\|
\big).\label{step1-estimate}
\end{eqnarray}

(\textbf{Step 2}) Estimate of
$Tr\Big(D^2_{xx}u_{1,n}(t,x,y)Q_{1,n}^\frac{1}{2}(Q_{1,n}^\frac{1}{2})^*\Big)$. \\
By differentiating twice with respect to  $x\in H^{(n)}$ in
$u_{1,n}(t,x,y)$, for any $x, h, k\in H^{(n)}$ we have
\begin{eqnarray*}
&&D_{xx}u_{1,n}(t,x,y)\cdot (h,
k)\\
&&=\int_0^{{+\infty}}\Big(D^2_{xx}(\bar{{F}}_n(x)-
\mathbb{E}{F}_n(x,Y_s^{x,n}(y)))\cdot (h,k),D_x\bar{u}_n(t,x)\Big)_{{H}}ds\\
&&+\int_0^{{+\infty}}\Big(D_x(\bar{{F}}_n(x)-
\mathbb{E}{F}_n(x,Y_s^{x,n}(y)))\cdot h,D^2_{xx}\bar{u}_n(t,x)\cdot
k\Big)_{{H} }ds\\
&&+\int_0^{{+\infty}}\Big(D_x(\bar{{F}}_n(x)-
\mathbb{E}{F}_n(x,Y_s^{x,n}(y)))\cdot k,D^2_{xx}\bar{u}_n(t,x)\cdot
h\Big)_{{H} }ds\\
&&+\int_0^{{+\infty}}\Big(\bar{{F}}_n(x)-\mathbb{E}{F}_n(x,Y_s^{x,n}(y)),
D^3_{xxx}\bar{u}_n(t,x)\cdot (h,k)\Big)_{{H}
}ds\\
&&:=\sum\limits_{i=1}^4J_{i,n}(t,x,y,h,k).
\end{eqnarray*}
By taking Lemma \ref{mix-derivative-2} and Lemma \ref{ux}   into
account, we can deduce
\begin{eqnarray}
&&|J_{1,n}(t,x,y,h,k)|\nonumber\\
&&\leq C\|h\|\cdot\|k\|\cdot\int_0^{+\infty}
e^{-c s} (1+\|x\|+\|y\|)ds\nonumber\\
&&\leq C\|h\|\cdot\|k\|(1+\|x\|+\|y\|). \label{J_1}
\end{eqnarray}
Next, thanks to  Lemma \ref{mix-derivative} and Lemma \ref{uxx} it
holds
\begin{eqnarray}
&&|J_{2,n}(t,x,y,h,k)|\nonumber\\
&&\leq C\|h\|\cdot\|k\| \cdot\int_0^{+\infty}
e^{-c s} (1+\|x\|+\|y\|)ds\nonumber\\
&&\leq C\|h\|\cdot\|k\|(1+\|x\|+\|y\|). \label{J_2}
\end{eqnarray}
Parallel to \eqref{J_2},  we can obtain the same estimate for
$J_{3,n}(t,x,y, h,k)$, that is,
\begin{eqnarray}
|J_{3,n}(t,x,y,h,k)| \leq C\|h\|\cdot\|k\|(1+\|x\|+\|y\|).
\label{J_3}
\end{eqnarray}
Thanks to Lemma \ref{uxxx} and \eqref{Averaging-Expectation}, we get
\begin{eqnarray}
&& |J_{4,n}(t,x,y,h,k)|\nonumber\\
&&\leq C\|h\|\cdot\|k\|
\cdot\int_0^{+\infty}e^{-\frac{\beta}{2} s} (1+\|x\|+\|y\|)ds\nonumber\\
&&\leq C\|h\|\cdot\|k\|(1+\|x\|+\|y\|). \label{J_4}
\end{eqnarray}
Collecting together \eqref{J_1}, \eqref{J_2}, \eqref{J_3} and
\eqref{J_4}, we obtain

\begin{eqnarray*}
|D^2_{xx}u_{1,n}(t,x,y)\cdot (h, k)|\leq C
\|h\|\cdot\|k\|(1+\|x\|+\|y\|),
\end{eqnarray*}
 so that, as the operator $Q_1$ has finite trace (see \eqref{Trace}), we get
\begin{eqnarray}
&&\left|Tr\Big(D^2_{xx}u_{1,n}(t,x,y)Q_{1,n}^\frac{1}{2}(Q_{1,n}^\frac{1}{2})^*\Big)\right|\nonumber\\
&&=\sum\limits_{k=1}^n \left|D^2_{xx}u_{1,n}(t,x,y)\cdot\Big(\sqrt{\lambda_{1,k}}e_k, \sqrt{\lambda_{1,k}}e_k\Big)\right|\nonumber\\
&&\leq C(1+\|x\|+\|y\|). \label{Tr}
\end{eqnarray}

Finally, by taking  inequalities \eqref{step1-estimate} and
\eqref{Tr} into account, we can conclude the proof of the lemma.
\end{proof}
\end{lemma}
As a consequence of Lemma \ref{4.3} and \ref{4.4}, we have the
following fact for the remainder term $r^\epsilon_n$.
\begin{lemma}\label{4.5}
Under the conditions of Lemma \ref{4.3}, for any  $r\in (0,
\frac{1}{4})$ we have
\begin{eqnarray*}
r_n^\epsilon(T,x,y)\leq C_{r,
T,\theta}\epsilon^{1-2r}(1+\|x\|^2+\|y\|^2+\|x\|^2_{(-A_n)^\theta}).
\end{eqnarray*}

\begin{proof}
By a variation of constant formula,   we have
\begin{eqnarray}
&&r_n^\epsilon(T,x,y)\nonumber\\
&&=\mathbb{E}[r_n^\epsilon(\delta_\epsilon,X^{\epsilon,n}_{T-\delta_\epsilon}(x,y),Y_{T-\delta_\epsilon}^{\epsilon,n}(x, y))]\nonumber\\
&&+\epsilon\mathbb{E}\left[\int^T_{\delta_\epsilon}(\mathcal{L}_2^nu_{1,n}-\frac{\partial
u_{1,n}}{\partial s})(s,
X^{\epsilon,n}_{T-s}(x,y),Y^{\epsilon,n}_{{T-s}}(x, y))
ds\right],\label{r-mild}
\end{eqnarray}
where $\delta_\epsilon\in (0,\frac{T}{2})$ is a constant, only
depending on $\epsilon>0$, to be chosen later. Now, we estimate the
two terms in the right hand side of \eqref{r-mild}. Firstly, note
that $u_n^\epsilon(0,x,y)=\bar{u}_n(0,x)$, it holds
\begin{eqnarray*}
r_n^\epsilon(\delta_\epsilon,
x,y)&=&u_n^\epsilon(\delta_\epsilon,x,y)-\bar{u}_n(\delta_\epsilon,x)-\epsilon
u_{1,n}(\delta_\epsilon,x,y)\\
&=&-\epsilon
u_{1,n}(\delta_\epsilon,x,y)+[u_n^\epsilon(\delta_\epsilon,x,y)-u_n^\epsilon(0,x,y)]\\
&&-[\bar{u}_n(\delta_\epsilon,x)-\bar{u}_n(0,x)].
\end{eqnarray*}
By lemma \ref{u-1-abso-lemma}, we have
\begin{eqnarray}
|\epsilon u_{1,n}(\delta_\epsilon,x,y)|\leq  C_T\epsilon
(1+\|x\|+\|y\|).\label{r-u1}
\end{eqnarray}
By using It\^{o}'s formula and taking  the expectation we obtain
\begin{eqnarray*}
&&u_n^\epsilon(\delta_\epsilon,x,y)-u_n^\epsilon(0,x,y)\\
&&=\mathbb{E}\int_0^{\delta_\epsilon}\phi'(X_s^{\epsilon,n}(x,y))\cdot[A_nX_s^{\epsilon,n}(x,y)+F_n(X_s^{\epsilon,n}(x,y),Y_s^{\epsilon,n}(x,y))]
ds\\
&&\quad+\frac{1}{2}\mathbb{E}\sum\limits_{k=1}^n\int_0^{\delta_\epsilon}\phi''(X_s^\epsilon(x,y))\cdot\left(\sqrt{\lambda_{1,k}}e_k,
\sqrt{\lambda_{1,k}}e_k\right).
\end{eqnarray*}
Then, due to Lemma \ref{moment-bound} and \ref{A-X}, for any $r\in
(0, \frac{1}{4})$ we have
\begin{eqnarray}
&&\left|u_n^\epsilon(\delta_\epsilon,x,y)-u_n^\epsilon(0,x,y)\right|\nonumber\\
&&\leq C\int_0^{\delta_\epsilon}
\big[\mathbb{E}\|A_nX_s^{\epsilon,n}(x,y)\|+1+\mathbb{E}\|X_s^{\epsilon,n}(x,y)\|+\mathbb{E}\|Y_s^{\epsilon,n}(x,y)\|\big]
ds\nonumber\\
&&\quad+C Tr(Q_1)\delta_\epsilon\nonumber\\
&&  \leq
C_{r,T}(\delta_\epsilon+\frac{\delta_\epsilon^\theta}{\theta}+\frac{\delta_\epsilon
}{\epsilon^r})(1+\|x\|_{(-A_n)^\theta}+\|x\|+\|y\|).  \label{r-u}
\end{eqnarray}
By using again It\^{o}'s formula, we get
\begin{eqnarray*}
&&\bar{u}_n(\delta_\epsilon,x)-\bar{u}_n(0,x)\\
&&=\mathbb{E}\int_0^{\delta_\epsilon}\phi'(\bar{X}_s^n
(x))\cdot[A_n\bar{X}_s^n(x)+\bar{F}_n(\bar{X}_s^n(x))]
ds\\
&&\quad+\frac{1}{2}\mathbb{E}\sum\limits_{k=1}^n\int_0^{\delta_\epsilon}\phi''(\bar{X}_s^n(x))\cdot\left(\sqrt{\lambda_{1,k}}e_k,
\sqrt{\lambda_{1,k}}e_k\right).
\end{eqnarray*}
Then, thanks to Lemma \ref{bar-X} and \eqref{bar-x-moment} it holds
\begin{eqnarray*}
&&\left|\bar{u}_n(\delta_\epsilon,x)-\bar{u}_n(0,x)\right|\\
&&\leq C\int_0^{\delta_\epsilon}
\big[\mathbb{E}\|A_n\bar{X}_s^n(x)\|+1+\mathbb{E}\|\bar{X}_s^n(x)\|\big]
ds\\
&&\quad+C Tr(Q_1)\delta_\epsilon\\
&&  \leq C_T(\delta_\epsilon+\frac{\delta_\epsilon^\theta}{\theta}
)(1+\|x\|_{(-A_n)^\theta}+\|x\|),
\end{eqnarray*}
which, in view of \eqref{r-u1} and \eqref{r-u}, means that
\begin{eqnarray*}
|r_n^\epsilon(\delta_\epsilon, x,y)|\leq
C_{r,T}(\epsilon+\delta_\epsilon+\frac{\delta_\epsilon^\theta}{\theta}+\frac{\delta_\epsilon
}{\epsilon^r})(1+\|x\|_{(-A_n)^\theta}+\|x\|+\|y\|),
\end{eqnarray*}
so that, due to Lemma \ref{moment-bound} and \ref{A-X},  {this
easily} implies that
\begin{eqnarray*}
&&\mathbb{E}[r_n^\epsilon(\delta_\epsilon,X^{\epsilon,n}_{T-\delta_\epsilon}(x,y),Y_{T-\delta_\epsilon}^{\epsilon,n}(x,
y))]\\
&&\leq
C_{r,T}(\epsilon+\delta_\epsilon+\frac{\delta_\epsilon^\theta}{\theta}+\frac{\delta_\epsilon
}{\epsilon^r})(1+\mathbb{E}\|X^{\epsilon,n}_{T-\delta_\epsilon}(x,y)\|_{(-A_n)^\theta}+\mathbb{E}\|Y^{\epsilon,n}_{T-\delta_\epsilon}(x,y)\|)\\
&&\leq
C_{r,T}(\epsilon+\delta_\epsilon+\frac{\delta_\epsilon^\theta}{\theta}+\frac{\delta_\epsilon
}{\epsilon^r})\big(|T-\delta_\epsilon|^{\theta-1}\|x\|_{(-A)^\theta}+(1+\|x\|+\|y\|)
(1+\frac{1}{\epsilon^{r}})\big)\\
&&\leq C_{r, T,
\theta}(\epsilon+\delta_\epsilon+\frac{\delta_\epsilon^\theta}{\theta}+\frac{\delta_\epsilon
}{\epsilon^r})(1+\frac{1}{\epsilon^{r}})(\|x\|_{(-A)^\theta}+1+\|x\|+\|y\|).
\end{eqnarray*}
If we pick $\delta_\epsilon=\epsilon^{\frac{1}{\theta}}\leq
\epsilon$, we get
\begin{eqnarray}
&&\mathbb{E}[r_n^\epsilon(\delta_\epsilon,X^{\epsilon,n}_{T-\delta_\epsilon}(x,y),Y_{T-\delta_\epsilon}^{\epsilon,n}(x,
y))]\nonumber\\
&&\leq C_{r, T,
\theta}\epsilon^{1-2r}(1+\|x\|_{(-A_n)^\theta}+\|x\|+\|y\|).\label{r-1-expectation}
\end{eqnarray}
Next, we estimate the second term in the right hand side of
\eqref{r-mild}. Thanks to Lemma \ref{4.3} and Lemma \ref{4.4}, we
have
\begin{eqnarray*}
&& \left|(\mathcal{L}_2^nu_{1,n}-\frac{\partial u_{1,n}}{\partial
s})(s,
X^{\epsilon,n}_{T-s}(x,y),Y_{T-s}^{\epsilon,n} (x,y))\right|\\
&&\leq C_T
 (1+\frac{1}{s}+s^{\theta-1})\big[1+\|X^{\epsilon,n}_{T-s}(x,y)\|+\|Y^{\epsilon,n}_{T-s}(x,y)\|\big]^2\\
&&\quad+C_T\|{A}_nX^{\epsilon,n}_{T-s}(x,y)\|\left(1+\|X^{\epsilon,n}_{T-s}(x,y)\|+\|X^{\epsilon,n}_{T-s}(x,y)\|\right),
\end{eqnarray*}
and, according to the previous Lemma \ref{moment-bound} and  Lemma
\ref{A-X}, this implies that

\begin{eqnarray*}
&&\epsilon\left|\mathbb{E}\left[\int^T_{\delta_\epsilon}(\mathcal{L}_2^nu_{1,n}-\frac{\partial
u_{1,n}}{\partial s})(s,
X^{\epsilon,n}_{T-s}(x,y),Y^{\epsilon,n}_{{T-s}}(x, y))
ds\right]\right|\\
&&\leq C_T\epsilon
\int^T_{\delta_\epsilon}(1+\frac{1}{s}+s^{\theta-1})\mathbb{E}\big[1
+\|X^{\epsilon,n}_{T-s}(x,y)\|^2+\|Y^{\epsilon,n}_{T-s}(x,y)\|^2\big] ds  \\
&&\quad+
C_{r,T}\epsilon\int_{\delta_\epsilon}^T\left[\mathbb{E}\|{A}_nX^{\epsilon,n}_{T-s}(x,y)\|^2\right]^\frac{1}{2}
\\
&&\qquad\qquad\qquad\cdot\left[\mathbb{E}(1+\|X^{\epsilon,n}_{T-s}(x,y)\|+\|X^{\epsilon,n}_{T-s}(x,y)\|)^2\right]^\frac{1}{2}ds\\
&&\leq C_{r,T}\epsilon(1+\|x\|^2+\|y\|^2+\|x\|^2_{(-A_n)^\theta})\\
&&\quad\quad\cdot \int^T_{\delta_\epsilon}(1+\frac{1}{\epsilon^r}+\frac{1}{s}+s^{\theta-1}+|T-s|^{\theta-1})ds\\
&&\leq C_{r,T} \epsilon(T+\frac{T^\theta}{\theta}+|\log T|
+|\log(\delta_\epsilon)|+\frac{T}{\epsilon^r})\\
&&\quad\quad\cdot(1+\|x\|^2+\|y\|^2+\|x\|^2_{(-A_n)^\theta})\\
&&\leq
C_{r,\theta,T}\epsilon(1+|\log\epsilon|+\frac{1}{\epsilon^r})(1+\|x\|^2+\|y\|^2+\|x\|^2_{(-A_n)^\theta})\\
&&\leq
C_{r,\theta,T}\epsilon^{1-r}(1+\|x\|^2+\|y\|^2+\|x\|^2_{(-A_n)^\theta}),
\end{eqnarray*}
which, together with  \eqref{r-1-expectation}, completes the proof.

\end{proof}

\end{lemma}

\subsection{Proof of {Theorem 3.1}}
Now we finish proof of main result introduced in Section 3
\begin{proof}
We stress that we  need only to prove \eqref{finite-weak}.  With the
notations introduced above, by Lemma \ref{u-0}, Lemma
\ref{u-1-abso-lemma} and Lemma \ref{4.5}, for any $r\in (0, 1)$
$x\in \mathscr{D}((-A)^{\theta})$ and $y\in H$ we have
\begin{eqnarray*}
&&\left|\mathbb{E}\phi(X^{\epsilon,n}_T(x^{(n)},y^{(n)}))-\mathbb{E}\phi(\bar{X}^{n}_T(x^{(n)}))\right|\\
&&=|u_n^\epsilon(T,x^{(n)},y^{(n)})-\bar{u}_n(T,x^{(n)}))|\\
&&=|u_{1,n}^\epsilon(T, x^{(n)},y^{(n)})|\epsilon+|r_n^\epsilon (T, x^{(n)},y^{(n)})|\\
&&\leq
C_{r,\theta,T}\epsilon^{1-r}(1+\|x^{(n)}\|^2+\|y^{(n)}\|^2+\|x^{(n)}\|^2_{(-A_n)^\theta})\\
&&\leq
C_{r,\theta,T}\epsilon^{1-r}(1+\|x\|^2+\|y\|^2+\|x\|^2_{(-A)^\theta}),
\end{eqnarray*}
where   $C_{r,\theta,T}$ is a constant independent of the dimension
  $n$. \\
The proof of Theorem \ref{theorem} is completed.
\end{proof}

\section{Appendix}
In this section, we state and prove some technical lemmas used in
the previous sections. We first study the differential dependence on
initial datum for the solution  $\bar{X}_t^n(x)$ of the averaged
system \eqref{aver-finite-equation}. In what follows we denote by
$\eta ^{h,x,n}_t$ the derivative of $\bar{X}_t^n(x)$ with respect to
$x$ along direction $h\in H^{(n)}$.

\begin{lemma}\label{eta-bound}
Under   (H.1) and (H.2), for any $x,h\in H^{(n)}$ and $T>0$ there
exits a constant $C_T>0$ such that for any $x, h\in H^{(n)}$,
\begin{eqnarray*}
\|\eta ^{h,x,n}_t\| \leq C_T\|h\|, \;t\in [0, T].
\end{eqnarray*}
\begin{proof}
Note that $\eta^{h,x,n}_t$ is the mild solution of  the first
variation equation associated with the problem
\eqref{aver-finite-equation}:
\begin{eqnarray*}
\begin{cases}
d\eta ^{h,x,n}_t=\left({A}_n\eta ^{h,x,n}_t+ \bar{{F}}'(\bar{X}_t^n(x))\cdot\eta ^{h,x,n}_t\right)dt,\\
\eta ^{h,x,n}_0=h.
\end{cases}
\end{eqnarray*}
This means that  $\eta ^{h,x,n}_t$  is the  solution of the integral
equation
\begin{eqnarray*}
\eta^{h,x,n}_t={S}_{t,n}h+\int_0^t {S}_{t-s,n}[
\bar{{F}}'_n(\bar{X}_s(x))\cdot\eta ^{h,x,n}_s]ds,
\end{eqnarray*}
and then, due to \eqref{bar-F-derivative} and  contractive property
of $S_{t,n}$, we get
\begin{eqnarray*}
\|\eta ^{h,x,n}_t\| \leq \|h\| +C\int_0^t\|\eta ^{h,x,n}_s\|ds.
\end{eqnarray*}
Then by  Gronwall lemma it follows that
\begin{eqnarray*}
\|\eta ^{h,x,n}_t\| \leq C_T\|h\|, \;t\in [0, T].
\end{eqnarray*}
\end{proof}
\end{lemma}

\begin{lemma}\label{eta-continuous}
Under the conditions of Lemma \ref{eta-bound}, for any $T>0$ and
$r\in (0,1)$ there exists a constant $C_{r, T }>0$ such that for any
$x,h\in H^{(n)}$ and $0<s\leq t\leq T$,
\begin{eqnarray*}
\|\eta_t^{h,x,n}-\eta_s^{h,x,n}\|\leq
C_{r,T}|t-s|^{1-r}(1+\frac{1}{s^{1-r}})\|h\|.
\end{eqnarray*}
\begin{proof}
See Proposition B.5 in \cite{Brehier}.
\end{proof}
\end{lemma}

\begin{lemma}\label{eta-regularity}
Under the conditions of Lemma \ref{eta-bound}, for any $T>0$  there
exists a constant $C_{T}>0$ such that for any $x, h\in H^{(n)}$,
\begin{eqnarray*}
\|A_n\eta^{h,x,n}_t\|\leq C_T (1+\frac{1}{t})(1+\|x\|)\|h\|,\;t\in
[0, T].
\end{eqnarray*}
\begin{proof}
See Proposition B.6 in \cite{Brehier}.
\end{proof}
\end{lemma}

After we have study the first order derivative of $\bar{X}_t^n(x)$,
we introduce the second order derivative of $\bar{X}_t^n(x)$ with
respect to $x$ in directions $h, k\in H^{(n)}$ denoted by
$\zeta^{h,k,x,n}$, which is the  solution of the second variation
equation
\begin{eqnarray}
\begin{cases}
d\zeta^{h,k,x,n}_t=\Big[{A}_n\zeta^{h,k,x,n}_t+
\bar{{F}}''_n(\bar{X}_t^n(x))\cdot(\eta^{h,x,n}_t,\eta^{k,x,n}_t)\\
\qquad\qquad\quad+\bar{{F}}'_n(\bar{X}_t^n(x))\cdot\zeta^{h,k,x,n}_t\Big]dt,\\
\zeta^{h,k,x}_0=0.
\end{cases}\label{variation-2}
\end{eqnarray}
\begin{lemma}\label{xi-2-bound}
Under the conditions of Lemma \ref{eta-bound}, for any $T>0$  there
exists a constant $C_{T}>0$ such that for any $x, h, k\in H^{(n)}$,
\begin{equation*}
\|\zeta^{h,k,x,n}_t\|\leq C_T\|h\|\cdot \|k\|, \; t\in [0, T].
\end{equation*}
\begin{proof}
See Proposition B.7 in \cite{Brehier}.
\end{proof}

\end{lemma}
We now introduce the regular results for $\bar{u}_n(t,x)$ defined
 in Section \ref{asym}.
\begin{lemma}\label{ux}
For any $T>0$, there exists a constant $C_T>0$ such that for any
$x\in {H}^{(n)}$ and $t\in [0, T]$, we have
$$\|D_{x}\bar{u}_n(t,x)\|\leq C_{T, \phi}.$$
\end{lemma}
\begin{proof}
Note that for any $t\in [0, T]$ and $h\in H^{(n)}$,
\begin{eqnarray*}
D_x\bar{u}_n(t,x)\cdot h
=\mathbb{E}\left(\phi'(\bar{X}_t^n(x)),\eta^{h,x,n}_t\right)_{H}.
\end{eqnarray*}
By Lemma \ref{eta-bound}, we have
\begin{eqnarray*}
|D_x\bar{u}_n(t,x)\cdot h|\leq C_T\sup\limits_{z\in H
}\|\phi'(z)\|\cdot\|h\|, \label{u-bar-deriv}
\end{eqnarray*}
so that
\begin{eqnarray*}
\|D_x\bar{u}_n(t,x)\| \leq C_{T, \phi}.\label{u-bar-deriv-1-1}
\end{eqnarray*}
\end{proof}

\begin{lemma}\label{uxx}
For any $T>0$, there exists a constant $C_{T, \phi}>0$ such that for
any $x,h,k\in H^{(n)}$ and $t\in [0, T]$, we have
$$\left|D^2_{xx}\bar{u}_n(t,x)\cdot(h,k)\right|\leq C_{T,\phi}\|h\|\cdot\|k\|.$$

\begin{proof}
For any $h, k \in H^{(n)}$, we have
\begin{eqnarray}
D^2_{xx}\bar{u}_n(t,x)\cdot(h,k)&=&\mathbb{E}\big[\phi''(\bar{X}_t^n(x))\cdot(\eta^{h,x,n}_t,\eta^{k,x,n}_t)\nonumber\\
&&+\phi'(\bar{X}_t^n(x))\cdot \zeta^{h,k,x,n}_t\big], \label{5.1}
\end{eqnarray}
where $\zeta^{h,k,x,n}$ is  governed by variation equation
\eqref{variation-2}. By invoking Lemma \ref{eta-bound} and Lemma
\ref{xi-2-bound}, we can get
\begin{eqnarray*}
|D^2_{xx}\bar{u}_n(t,x)\cdot(h,k)|\leq C_{T, \phi} \|h\|\cdot\|k\|.
\end{eqnarray*}
\end{proof}
\end{lemma}
By proceeding again as in the proof of above lemma,  we have the
following result.
\begin{lemma}\label{uxxx}
For any $T>0$, there exists $C_T>0$ such that for any $x,h,k,l\in
H^{(n)}$ and $t\in [0, T]$, we have
$$D^3_{xxx}\bar{u}_n(t,x)\cdot(h,k,l)\leq C_{T,\phi}\|h\|\cdot\|k\|\cdot\|l\|.$$
\end{lemma}

Finally, we introduce some regular results  which is crucial in
order to prove some important estimates in Section \ref{asym}.

\begin{lemma}\label{mix-derivative}
There exist constants $C, c>0$ such that for any $x, y, h\in
H^{(n)}$ and $t>0$ it holds
\begin{eqnarray*}
\|D_x (\bar{F}_n(x)-\mathbb{E}F_n(x, Y^x_t(y)))\cdot h\| \leq
Ce^{-ct}\|h\| \left(1+\|x\| +\|y\| \right).
\end{eqnarray*}
\begin{proof}
We shall follow the approach of \cite[Proposition C.2]{Brehier}. For
any $t_0>0$, we set
\begin{eqnarray*}
\tilde{F}_{t_0,n}(x,y,t)=\hat{F}_n(x,y,t)-\hat{F}_n(x,y,t+t_0),
\end{eqnarray*}
where
\begin{eqnarray*}
\hat{F}_n(x,y,t):=\mathbb{E}F_n(x, Y^{x,n}_t(y)).
\end{eqnarray*}
Thanks to Markov property we may write that
\begin{eqnarray*}
\tilde{F}_{t_0,n}(x,y,t)&=&\hat{F}_n(x,y,t)-\mathbb{E}F_n(x,Y_{t+t_0}^{x,n}(y))\\
&=&\hat{F}_n(x,y,t)-\mathbb{E}\hat{F}_n(x, Y_{t_0}^{x,n}(y),t).
\end{eqnarray*}
In view of the assumption (H.1), $\hat{F}_n$ is
G\^{a}teaux-differentiable with respect to $x$ at $(x,y,t)$.
Therefore, we have for any $h\in H^{(n)}$ that
\begin{eqnarray}
D_x\tilde{F}_{t_0,n}(x,y,t)\cdot h&=&D_x\hat{F}_n(x,y,t)\cdot
h-\mathbb{E}D_x\left(\hat{F}_n(x, Y_{t_0}^{x,n}(y),t)\right)\cdot h\nonumber\\
&=&\hat{F}_{n,x}'(x,y,t)\cdot h-\mathbb{E}\hat{F}_{n,x}'(x,
Y_{t_0}^{x,n}(y),t)\cdot
h\nonumber\\
&&-\mathbb{E}\hat{F}_{n,y}'(x,
Y_{t_0}^{x,n}(y),t)\cdot\left(D_xY_{t_0}^{x,n}(y)\cdot
h\right),\label{7-1-1}
\end{eqnarray}
where we use the symbol $\hat{F}_{n,x}'$ and $\hat{F}_{n,y}'$ to
denote the  derivative with respect to $x$ and $y$, respectively.
Note that the first derivative $\varsigma_t^{x,y,
h,n}=D_xY_{t}^{x,n}(y)\cdot h$, at the point $x$ and along the
direction $h\in H^{(n)}$, is the solution of variation equation
\begin{eqnarray*}
d\varsigma_t^{x, y, h,n}&=&\left(A_n\varsigma_t^{x,
y,h,n}+G_{n,x}'(x, Y_t^{x,n}(y))\cdot h+G_{n,y}'(x,
Y_t^{x,n}(y))\cdot\varsigma_t^{x, y,h,n}\right)dt
\end{eqnarray*}
with initial data $\varsigma_0^{x,y, h,n}=0$. Hence, thanks to
{(H.2)}, it is immediate to check that for any $t\geq 0$,
\begin{eqnarray}
\mathbb{E}\|\varsigma_t^{x,y, h,n}\|\leq C\|h\|.\label{7-2}
\end{eqnarray}
Note that there exists a constant $c>0$, such that, for any $y_1,
y_2\in H^{(n)}$, it holds
\begin{eqnarray}
\|\hat{F}_n(x,y_1, t)-\hat{F}_n(x,y_2,t)\|&=&\|\mathbb{E}F_n(x,
Y_t^{x,n}(y_1))-\mathbb{E}F_n(x,
Y_t^{x,n}(y_2))\|\nonumber\\
&\leq&C\mathbb{E}\|Y_t^{x,n}(y_1)-Y_t^{x}(y_2)\|\nonumber\\
&\leq& Ce^{-ct}\|y_1-y_2\|,\nonumber
\end{eqnarray}
which implies
\begin{eqnarray}
\|\hat{F}_{n,y}'(x, y,t)\cdot k\|\leq Ce^{-ct}\|k\|,\; k\in
H.\label{7-3}
\end{eqnarray}
Therefore, thanks to \eqref{7-2} and \eqref{7-3}, we can conclude
that
\begin{eqnarray}
\|\mathbb{E}[\hat{F}_{n,y}'(x,
Y_{t_0}^{x,n}(y),t)\cdot\left(D_xY_{t_0}^{x,n}(y)\cdot
h\right)]\|\leq C e^{-ct}\|h\|.\label{7-4}
\end{eqnarray}
Then, we directly have
\begin{eqnarray}
&&\hat{F}_{n,x}'(x,y_1,t)\cdot h-\hat{F}_{n,x}'(x,y_2,t)\cdot h\nonumber\\
&&\quad=\mathbb{E}\left(F_{n,x}'(x, Y_t^{x,n}(y_1))\right)\cdot
h-\mathbb{E}\left(F_{n,x}'(x, Y_t^{x,n}(y_2))\right)\cdot h\nonumber\\
&&\quad\quad+\mathbb{E}\left(F_{n,y}'(x, Y_t^{x,n}(y_1))\cdot
\varsigma_t^{x,y_1,
h,n}-F_{n,y}'(x, Y_t^{x,n}(y_2))\cdot \varsigma_t^{x,y_2, h,n}\right)\nonumber\\
&&\quad= \mathbb{E}\left(F_{n,x}'(x, Y_t^{x,n}(y_1))\right)\cdot
h-\mathbb{E}\left(F_{n,x}'(x, Y_t^{x,n}(y_2))\right)\cdot h\nonumber\\
&&\quad\quad+\mathbb{E}\left([F_{n,y}'(x,
Y_t^{x,n}(y_1))-F_{n,y}'(x,
Y_t^{x,n}(y_2))]\cdot\varsigma_t^{x,y_1, h,n} \right)\nonumber\\
&&\quad\quad+\mathbb{E}\left(F_{n,y}'(x, Y_t^{x,n}(
y_2))\cdot(\varsigma_t^{x,y_1, h,n}-\varsigma_t^{x,y_2,
h,n})\right).\label{7-5}
\end{eqnarray}
{First it is easy to show}
\begin{eqnarray}
&&\|\mathbb{E}\left(F_{n,x}'(x, Y_t^{x,n}(y_1))\right)\cdot
h-\mathbb{E}\left(F_{n,x}'(x, Y_t^{x,n}(y_2))\right)\cdot h\|\nonumber\\
&&\quad\leq\mathbb{E}\|\left(F_{n,x}'(x, Y_t^{x,n}(y_1))\right)\cdot
h-\left(F_{n,x}'(x, Y_t^{x,n}(y_2))\right)\cdot h\|\nonumber\\
&&\quad\leq
C\mathbb{E}\|Y_t^{x,n}(y_1)-Y_t^{x,n}(y_2)\|\cdot\|h\|\nonumber\\
&&\quad\leq Ce^{-ct}\|y_1-y_2\|\cdot\|h\|.\label{7-6}
\end{eqnarray}
Next, by  Assumption (H.2)  we have
\begin{eqnarray}
&&\|\mathbb{E}\left([F_{n,y}'(x, Y_t^{x,n}(y_1))-F_{n,y}'(x,
Y_t^{x,n}(y_2))]\cdot\varsigma_t^{x,y_1, h,n} \right)\|\nonumber\\
&&\quad\leq\mathbb{E}\|[F_{n,y}'(x, Y_t^{x,n}(y_1))-F_{n,y}'(x,
Y_t^{x,n}(y_2))]\cdot\varsigma_t^{x,y_1, h,n}\|\nonumber\\
&&\quad\leq C\{\mathbb{E}\|\varsigma_t^{x,y_1,
h,n}\|^2\}^{\frac{1}{2}}\cdot\{\mathbb{E}\|Y_t^{x,n}(y_1)-Y_t^{x,n}(y_2)\|^2\}^{\frac{1}{2}}\nonumber\\
&&\quad\leq C e^{-ct}\|h\|\cdot\|y_1-y_2\|.\label{7-7}
\end{eqnarray}
By making use of  {Assumption (H.1)} again, we can show that there
exists a constant $c'>0$  such that one has
\begin{eqnarray}
&&\|\mathbb{E}\left(F_{n,y}'(x, Y_t^{x,n}
(y_2))\cdot(\varsigma_t^{x,y_1,
h,n}-\varsigma_t^{x,y_2, h,n})\right)\|\nonumber\\
&&\quad\leq\mathbb{E}\|\left(F_{n,y}'(x, Y_t^{x,n}(
y_2))\cdot(\varsigma_t^{x,y_1, h,n}-\varsigma_t^{x,y_2, h,n})\right)\|\nonumber\\
&&\quad\leq C\mathbb{E}\|\varsigma_t^{x,y_1, h,n}-\varsigma_t^{x,y_2, h,n}\|\nonumber\\
&&\quad\leq C e^{-c't}\|y_1-y_2\|\cdot\|h\|.\label{7-8}
\end{eqnarray}
Collecting together \eqref{7-5}, \eqref{7-6}, \eqref{7-7} and
\eqref{7-8}, we get
\begin{eqnarray*}
&&\|\hat{F}_{n,x}'(x,y_1,t)\cdot h-\hat{F}_{n,x}'(x,y_2,t)\cdot h\|\nonumber\\
&&\leq C e^{-c_0t}\|y_1-y_2\|\cdot\|h\|,
\end{eqnarray*}
which means
\begin{eqnarray}
&&\|\hat{F}_{n,x}'(x,y,t)\cdot h-\mathbb{E}\hat{F}_{n,x}'(x,Y^{x,n}_{t_0}(y),t)\cdot h\|\nonumber\\
&&\leq C e^{-c_0t}(1+\|y\|)\cdot\|h\| \label{7-9}
\end{eqnarray}
since
\begin{eqnarray*}
\mathbb{E}\|Y^{x,n}_{t_0}(y)\|\leq C(1+\|x\|+\|y\|).
\end{eqnarray*}
Returning to \eqref{7-1-1}, by \eqref{7-4} and \eqref{7-9} we
conclude that
\begin{eqnarray*}
\|D_x\tilde{F}_{t_0,n}(x,y,t)\cdot h\|\leq
Ce^{-ct}(1+\|x\|+\|y\|)\|h\|.
\end{eqnarray*}
 By taking the limit as $t_0$ goes to infinity we obtain
 \begin{eqnarray*}
 \|D_x (\bar{F}_n(x)-\mathbb{E}F_n(x, Y^x_t(y)))\cdot h\| \leq
Ce^{-ct}\|h\| \left(1+\|x\| +\|y\| \right).
\end{eqnarray*}
\end{proof}
\end{lemma}
Proceeding with similar arguments above we can obtain similar result
concerning the second  order differentiability.
\begin{lemma}\label{mix-derivative-2}
There exist constants $C, c>0$ such that for any $x, y, h, k\in
H^{(n)}$ and $t>0$ it holds
\begin{eqnarray*}
\|D^2_{xx} (\bar{F}_n(x)-\mathbb{E}F_n(x, Y^{x,n}_t(y)))\cdot (h,
k)\| \leq Ce^{-ct}\|h\| \cdot \|k\|\left(1+\|x\| +\|y\| \right).
\end{eqnarray*}
\end{lemma}

\section*{Acknowledgments}
We would like to thank Professor Dirk Bl\"{o}mker for helpful
discussions and comments. Hongbo Fu is supported by CSC scholarship
(No. [2015]5104), NSF  of China (Nos. 11301403) and Foundation of
Wuhan Textile University 2013. Li Wan is supported by NSF  of China
(No. 61573011) and Science and Technology Research Projects of Hubei
Provincial Department of Education (No. D20131602). Jicheng Liu is
supported by NSF  of China (No. 11271013). Xianming Liu is supported
by NSF  of China (No. 11301197)

\label{}










\end{document}